\documentclass{amsart}

\usepackage{amsmath,amsthm,amsfonts,amssymb,oldgerm}
\usepackage[frenchb,english]{babel}
\usepackage{indentfirst}
\numberwithin{equation}{section}
\usepackage[citecolor=blue,colorlinks=true]{hyperref}
\usepackage{geometry,color}
\allowdisplaybreaks[1]
\usepackage{mathrsfs}

\newtheorem{theorem}{Theorem}[section]
\newtheorem{proposition}[theorem]{Proposition}

\newtheorem{lemma}[theorem]{Lemma}
\newtheorem{remark}[theorem]{Remark}

\theoremstyle{definition}
\newtheorem{definition}[theorem]{Definition}

\renewcommand\d{\partial}
\def\eps{\varepsilon }
\newcommand{\dd}{\mathrm{d}}
\renewcommand{\div}{\textrm{div}}

\newcommand{\Id}{{\rm Id }}

\newcommand{\E}{\mathbb{E}}
\newcommand{\N}{\mathbb{N}}

\newcommand{\R}{\mathbb{R}}
\newcommand{\T}{\mathbb{T}}
\newcommand{\Z}{\mathbb{Z}}
\newcommand{\PP}{\mathbb{P}}

\newcommand\cA{{\mathcal A}}
\newcommand\cB{{\mathcal B}}
\newcommand\cF{{\mathcal F}}
\newcommand\cG{{\mathcal G}}

\newcommand\cM{{\mathcal M}}
\newcommand\cN{{\mathcal N}}

\newcommand\cQ{{\mathcal Q}}

\newcommand{\Pl}{\Pi_{\textrm{loc}}}
\newcommand\dM{\cM^{\frac12}}
\newcommand{\rinf}{\rho_{\infty}}

\DeclareMathOperator{\Tr}{Tr}

\def\<{\langle}
\def\>{\rangle}
\def\refe#1{\eqref{#1}}
\def\L{\mathscr{L}}

\title{Invariant measures\\ for a stochastic Fokker-Planck equation}

\author{Sylvain De Moor}
\address{\'ENS Cachan-Antenne de Bretagne, Bruz, France}
\email{sylvain.demoor@bretagne.ens-cachan.fr}
\thanks{Research of Sylvain De Moor was partially supported by the ANR project
STOSYMAP}
 
\author{L. Miguel Rodrigues}
\address{Universit\'e Lyon 1 \& INRIA, Villeurbanne, France}
\email{rodrigues@math.univ-lyon1.fr}
\thanks{Research of L. Miguel Rodrigues was partially supported by the ANR project
BoND ANR-13-BS01-0009-01}

\author{Julien Vovelle}
\address{CNRS \& Universit\'e Lyon 1, Villeurbanne, France}
\email{vovelle@math.univ-lyon1.fr}
\thanks{Research of Julien Vovelle was partially supported by the ANR project
STOSYMAP and the ANR project STAB}

\begin{document}
\begin{abstract}
\footnotesize We study a kinetic Vlasov/Fokker-Planck equation perturbed by a stochastic forcing term. When the noise intensity is not too large, we solve the corresponding Cauchy problem in a space of functions ensuring good localization in the velocity variable. Then we show under similar conditions that the generated dynamics, with prescribed total mass, admits a unique invariant measure which is exponentially mixing. The proof relies on hypocoercive estimates and hypoelliptic regularity. At last we provide an explicit example showing that our analytic framework does require some smallness condition on the noise intensity.
\vspace{0.1cm}

\end{abstract}

\date{\today}
\maketitle

\tableofcontents

\section{Introduction}

\noindent We are interested in the large-time dynamics generated by the following stochastic Fokker-Planck equation
\begin{equation}\label{FP}
\dd f\ +\ v\cdot\nabla_x f\ \dd t\ +\ \lambda \nabla_vf\odot\dd W_t\ =\ \cQ(f)\ \dd t.
\end{equation}
The unknown $f$ depends on a time variable $t\in[0,\infty)$, a space variable $x\in\T^N$, a velocity $v\in\R^N$ and the alea. The operator $\cQ$ is the Fokker-Planck operator whose expression is given by
\begin{equation}\label{defQQ}
\cQ(f)\ =\ \Delta_v f\ +\ \div_v(vf).
\end{equation}
As is customary for stochastic dynamics our focus is on the existence of invariant measures and exponential mixing.

Before entering into the heart of the analysis, since stochastic kinetic modeling is not a widespread practice we first provide some elements of justification and explanation for the equation itself.

\bigskip

\paragraph{\textbf{Singular Vlasov Force Term}} 
The classical deterministic Vlasov-Fokker-Planck equation 
\begin{equation}\label{FPVlasov}
\partial_t f\ +\ v\cdot\nabla_x f\ +\ F(t,x)\cdot \nabla_v f\ =\ \cQ(f),
\end{equation}
with $\cQ$ given by \eqref{defQQ}, is the evolution equation for the density $(x,v)\mapsto f(t,x,v)$ with respect to the Lebesgue measure on $\T^N\times\R^N$ of the law of the process $(x_t,v_t)$ solution to the stochastic differential system
\begin{equation}\label{XVsto}
\begin{cases}
\,\dd x_t&=\ v_t\dd t\,,\\
\,\dd v_t&=\ -F(t,x_t)\dd t-v_t\dd t+\sqrt{2}\dd\hat B_t.
\end{cases}
\end{equation}
Here, $(\hat B_t)_{t\geq 0}$ is an $N$-dimensional Brownian motion on a probability space $(\hat\Omega,\hat{\mathcal{F}},\hat\PP)$ and its presence in the system above accounts for the interaction with a common thermal bath. We are interested in the situation where the force $F(t,x)$ in \eqref{FPVlasov} is so singular that a stochastic modeling is more appropriate. We focus on the case where the singularity is of a very specific type, namely when it is given by a time-white noise. This means that there will be two different noise terms in \eqref{XVsto}. This situation is also encountered in mean field games, see for example Equation~(2) in \cite{CardaliaguetDelarueLasryLions2015}\footnote{In our context, this is $\dd\hat B_t$ that is the common noise.}. The notation $\odot$ in \eqref{FP} emphasizes the scalar product in $\R^N$ and the fact that we consider the stochastic term in the Stratonovich sense\footnote{See Appendix~\ref{app:ItoInt}.} \cite[Chapter~20]{SchillingPartzsch2014}. This is indeed more natural, \cite{WongZakai1965}, when the singular force term $\dd W_t$ arises from the singular limit of more regular force terms $F(t,x)\dd t$, as in \eqref{FPVlasov}. To some extent our modeling considerations are similar to the ones leading to the introduction of a stochastic force in the Navier-Stokes equations but the nature of the description, kinetic rather than macroscopic, leads to a very different equation where the noise is multiplicative and in front of a derivative instead of being simply additive.

\bigskip

From a modeling point of view it would be more satisfactory to consider a force that instead of being purely singular would be the sum of a time-white noise and of a part that would be given as a smooth deterministic function of the density of the law. This would result in a nonlinear stochastic partial differential equation. Yet our goal is to study the large-time dynamics induced by \eqref{FP} and we stress that even in a purely deterministic setting ($\lambda=0$) where the large-time dynamics is trivial --- in the sense that it leads at exponential rate to the convergence towards a unique stationary solution --- the quantitative analysis of simplest relevant models is a very recent achievement, see \cite{HwangJang13} and \cite{HerauThomann16}. Moreover the foregoing analyses are actually restricted by (non explicit) weakly nonlinear assumptions. An analogous stochastic nonlinear analysis appears thus as far beyond reach of a first investigation on our class of problem. We believe however that we could add deterministic force terms preserving linearity of \eqref{FP} with almost immaterial modifications.

\bigskip

\paragraph{\textbf{Time-white noise}} 
To be more specific concerning what a time-white noise is, let us first recall the L{\'e}vy-Ciesielski construction of the Brownian motion on $[0,1]$, \cite[Section~3.2]{SchillingPartzsch2014}. Let $(\Omega,\mathcal{F},\PP)$ be a probability space and $(a_n)$ be a sequence of independent centered normalized Gaussian random variables. Also let $(H_n)$ denote the Haar basis of $L^2(0,1)$. Then (\textit{cf.} \cite[Formula~(3.1)]{SchillingPartzsch2014}) the formula
$$
\tilde{\beta}(t)=\sum_{n=0}^\infty a_n \<\mathbf{1}_{[0,t]},H_n\>_{L^2(0,1)}
$$
defines a Brownian motion on $[0,1]$. The (constant in space, time) white noise is then
\begin{equation}\label{LevyCiesielsky}
\frac{\dd\;}{\dd t}\tilde{\beta}(t)=\sum_{n=0}^\infty a_n H_n(t).
\end{equation}
Note that, as suggested by the terminology, the white noise sums all frequencies, roughly encoded by the label $n$, with equal strength. Though equally interesting from a modeling point of view, a space-time white noise seems too singular to be handled in \eqref{XVsto} by currently available techniques in the analysis of stochastic partial differential equations. We will restrict instead to a time-white noise which is colored in space.

To do so, more generally, let $(G_j)_{j\in\N}$ be an orthonormal basis of $L^2(\T^N;\R^N)$, let $(\beta_j(t))_{j\in\N}$ be a sequence of independent Brownian motion on $\R_+$ and let $\Gamma$ be a self-adjoint non-negative trace-class operator on $L^2(\T^N;\R^N)$. We set
\begin{equation}\label{colouredNoise1}
W_t(x)\ :=\ \sum_{j\geq 0} \Gamma^{\frac{1}{2}}G_j(x)\ \beta_j(t) 
\end{equation}
so that $W_t$ is well defined as an $L^2(\T^N;\R^N)$-valued process. Indeed, using independence of the $(\beta_0,\beta_1,\ldots)$,
\begin{equation}\label{WienerHS}
\E\|W_t\|_{L^2(\T^N;\R^N)}^2=t\,\Tr(\Gamma)
\end{equation}
is then finite for every $t$. Alternatively $W_t$ can be written $W_t=\Gamma^{\frac{1}{2}} W^*(t)$ where $W^*$ is the cylindrical Wiener process~\cite[Section~4.3.1]{daprato}
\begin{equation}\label{cylindricalW}
W^*(t):=\sum_{j\geq 0}\beta_j(t) G_j\,.
\end{equation}
The process $W^*(t)$ is well defined as a $\mathfrak{U}$-valued process, where $\mathfrak{U}$ is any Hilbert space such that the injection $L^2(\T^N)\rightharpoonup \mathfrak{U}$ is Hilbert-Schmidt~\cite[Section~4.3.1]{daprato} since by the same arguments as for \eqref{WienerHS}, $\E\|W^*(t)\|^2_{\mathfrak{U}}$ is finite for every $t$. Therefore $W_t$ will define a $\Gamma$-Wiener process on $L^2(\T^N;\R^N)$, \cite[Section~4.1]{daprato}. Then the time-derivative of $W_t$, which we write formally as
\begin{equation}\label{ddtcylindricalW}
\frac{\dd\;}{\dd t}W^*(t):=\sum_{j\geq 0}\frac{\dd\;}{\dd t}\beta_j(t) \Gamma^{\frac12}G_j.
\end{equation}
is the corresponding time white noise on $\R_+\times\T^N$. 

Note that it would be harmless to assume that $(G_j)$ is a basis of eigenvectors of $\Gamma$, that is, $\Gamma G_j = \gamma_jG_j$, where $(\gamma_j)$ is a sequence of non-negative numbers in $l^1(\N)$. Then by expanding each $\beta_j$ under the form \eqref{LevyCiesielsky} and noticing that the elements $(t,x)\mapsto H_n(t)G_j(x)$, $n,j\in\N$ constitute an orthonormal basis of a time-space $L^2$, definition \eqref{ddtcylindricalW} appears as a generalization of \eqref{LevyCiesielsky} where amplitudes are of equal strength, independent and Gaussian with respect to the time-frequency $n$ but (deterministic and) damped with factor $\gamma_j^{1/2 }$ with respect to the space-frequency $j$. Hence the terminology above : white in time, colored in space.

In what follows, we set $F_j =\Gamma^{\frac{1}{2}}G_j$ and write $W_t$ in the form
$$
W_t(x)\ =\ \sum_{j\geq 0} F_j(x)\ \beta_j(t).
$$
We assume the following additional regularity in space of the noise: 
\begin{equation}\label{noise}
\textrm{each }F_j\textrm{ is }\mathcal{C}^1\qquad\textrm{ and }\qquad
\sum_{j\geq 0}(\|F_j\|^2_{\infty}+\|\nabla_xF_j\|^2_{\infty})\leq 1.
\end{equation}
Since we fix the intensity of the noise to the value $1$, it is the parameter $\lambda$ that will measure the strength of the noise term in \eqref{FP}. 
\bigskip

\paragraph{\textbf{Solving \eqref{FP} in $L^2$}}

Before dealing with large-time behavior, we first provide what we believe to be the most natural existence result for \eqref{FP}.

For this purpose, we fix a filtration $(\mathcal{F}_t)_{t\geq 0}$ on $(\Omega,\mathcal{F},\PP)$. We assume that $(\mathcal{F}_t)_{t\geq 0}$ satisfies the usual regularity condition: it is right continuous and $\mathcal{F}_0$ contains all the
$\PP$-–null sets of $\mathcal{F}$. We recall\footnote{See Appendix~\ref{app:ItoInt}.} that a process $(X(t))_{t\in[0,T]}$ with values in a measure space is said to be adapted if, for each $t\in[0,T]$, $X(t)$ is a random variable on $(\Omega,\mathcal{F}_t)$.

\begin{theorem}\label{thmexFP}
Suppose that hypothesis $(\ref{noise})$ is satisfied and consider $\lambda\in\R$ and an initial datum
$$
f_\mathrm{in}\in L^2(\Omega;L^2(\T^N\times\R^N)). 
$$
Then there exists an adapted process $(f(t))_{t\geq 0}$ on $L^2(\Omega;L^2((\T^N\times\R^N)))$ which satisfies
\begin{itemize}
\item[$(i)$]\label{SolFPItem1} for any $T>0$, $f\in L^2(\Omega;\mathcal{C}_w([0,T];L^2(\T^N\times\R^N))))$;
\item[$(ii)$]\label{SolFPItem2} for any $T>0$, $\nabla_v f\in L^2(\Omega\times [0,T];L^2(\T^N\times\R^N)))$;
\item[$(iii)$]\label{SolFPItem4} for any $\varphi$ in $\mathcal{C}^\infty_c(\T^N\times\R^N)$ and any $t\geq 0$,
\begin{equation}\label{SolutionFP}
\begin{array}{lll}
& \displaystyle \langle f(t),\varphi\rangle  &=\ \displaystyle \langle f_\mathrm{in},\varphi\rangle  + \int_0^t \langle f(s), v\cdot\nabla_x \varphi\rangle \dd s + \lambda\sum_{j\geq 0}\int_0^t \left\langle f(s),F_j\cdot\nabla_v \varphi\right\rangle \dd\beta_j(s) \\
&& \displaystyle+\int_0^t\langle f(s),\cQ^*(\varphi)\rangle \dd s +\frac{\lambda^2}{2}\sum\limits_{j\geq 0} \int_0^t \left\langle  f(s),\left(F_j\cdot\nabla_v \right)^2\varphi\right\rangle \dd s, \quad \text{a.s.},
\end{array}
\end{equation}
\end{itemize}
where $\cQ^*$ is the formal adjoint to $\cQ$ defined by \eqref{AdjointQ}. Moreover the solution $f$ is unique and satisfies the estimate
\begin{equation}\label{EEstimateFPsolutions}
\E\|f(t)\|^2_{L^2(\T^N\times\R^N)}\leq e^{tN}\E\|f_\mathrm{in}\|^2_{L^2(\T^N\times\R^N)}
\end{equation}
for all $t\in[0,T]$.
\end{theorem}

\noindent In the foregoing statement, $\mathcal{C}_w([0,T];E)$ denotes functions that are continuous when the normed space $E$ is endowed with its weak topology. See details in Section~\ref{sec:exuniq}.

\medskip

Since our main concern is long-time behavior the proof of Theorem~\ref{thmexFP} is delayed to Appendix~\ref{app:thmexFP}. The proof is rather classical but still quite technical, the hardest part being probably the uniqueness part. It may be worth mentioning that our uniqueness result contains the deterministic case ($\lambda=0$) and with this respect improves on the one contained in \cite[Appendix~A.20]{villani} but is still far from reaching best expected results (uniqueness in spaces allowing for Gaussian growth at infinity) that may be proved by inspecting existence results for the dual equation.

\bigskip

Unfortunately the foregoing statement is essentially useless on large time. Indeed the bound \eqref{EEstimateFPsolutions} is sharp for solutions starting from $f_\mathrm{in}\in L^2(\Omega\times\T^N\times\R^N)$ as is easily seen by considering the case when $\lambda=0$ and $f_\mathrm{in}$ is independent of $x$. In this special case, the dynamics reduces to the evolution generated by $\cQ$ (on spatially homogeneous functions) and it is well-known that a higher localization in the velocity variable is indeed needed to prevent exponential growth in time and an even stronger localization to reach exponential convergence at natural decay rates. See \textit{e.g.} \cite[Appendix~A]{Gallay-Wayne-invariant_manifold} where it is proved that for the corresponding evolution, in the scale of spaces $L^2((1+|v|^2)^{m}\dd v)$, $m=N/2$ is the threshold for boundedness and $m=N/2+1$ is the threshold for convergence at largest possible rate. 

\bigskip

\paragraph{\textbf{Solving \eqref{FP} in weighted spaces}}

Our second result provides the missing localization as it is concerned by the resolution of the Cauchy problem for \eqref{FP} in the weighted space $L^2(\T^N\times\R^N,\dd x\times\cM^{-1}\dd v)$, where 
\begin{equation}\label{def:maxwellian}
\cM(v)=(2\pi)^{-N/2}e^{-|v|^2/2},\quad v\in\R^N,
\end{equation}
is the Max\-wel\-lian distribution on $\R^N$. Maxwellian is the usual terminology in the study of kinetic equations, Gaussian distribution is of course the usual term in probability theory. At any rate, the relevance of $\cM$ here originates in the fact it is a stationary solution to \eqref{FP} in the case $\lambda= 0$ (no stochastic forcing). Note also that the operator $\cQ$ is self-adjoint in the weighted space $L^2(\R^N,\cM^{-1}\dd v)$, as it is already self-adjoint on $L^2(\cM^{-1}\dd v)$ when restricted to spatially homogeneous functions. In an equivalent manner, we will state our conclusions in $L^2$ for the new unknown
$$
g=\cM^{-\frac12}f.
$$
Then $g$ should solve
\begin{equation}\label{eqg}
\left\{
\begin{array}{ll}
& \dd g\ +\ v\cdot\nabla_x g\ \dd t\ +\ \lambda\left(\nabla_v-\dfrac{v}{2}\right)g\ \odot\dd W_t\ =\ Lg\ \dd t\\
& g(0)=g_\mathrm{in}
\end{array} \right.
\end{equation}
with
\begin{equation}\label{defLprem}
Lg\ =\ \Delta_v g\ +\ \left(\dfrac{N}{2}-\dfrac{|v|^2}{4}\right)g,
\end{equation}
if $f$ is solution to \eqref{FP} with initial datum $f_\mathrm{in}=\cM^{\frac12} g_\mathrm{in}$. Note that as expected the operator $L$ is a self-adjoint operator on $L^2(\R^N)$, indeed it is the quantum harmonic oscillator operator. Our second result is the following one.
 
\begin{theorem}\label{thmex}
Suppose that hypothesis $(\ref{noise})$ holds and let 
$$g_\mathrm{in}\in L^2(\Omega;L^2(\T^N\times\R^N)).$$
For any $|\lambda|< 1$, there exists a unique adapted process $(g(t))_{t\geq 0}$ on $L^2(\Omega;L^2(\T^N\times\R^N))$ which satisfies
\begin{itemize}
\item[$(i)$] for any $T>0$, $g\in L^2(\Omega;\mathcal{C}_w([0,T];L^2(\T^N\times\R^N)))$;
\item[$(ii)$] for any $T>0$, $\left(\nabla_v+\dfrac{v}{2}\right)g\in L^2(\Omega\times(0,T); L^2(\T^N\times\R^N))$;
\item[$(iii)$] for any $\varphi$ in $\mathcal{C}^\infty_c(\T^N\times\R^N)$ and any $t\geq 0$,
\begin{equation}\label{Solution}
\begin{array}{lll}
& \displaystyle \langle g(t),\varphi\rangle&=\ \displaystyle \langle g_\mathrm{in},\varphi\rangle  + \int_0^t \langle g(s), v\cdot\nabla_x \varphi\rangle \dd s + \lambda\sum_{j\geq 0}\int_0^t \left\langle g(s),F_j\cdot \left(\nabla_v+\dfrac{v}{2}\right) \varphi\right\rangle \dd\beta_j(s) \\
&& \displaystyle+\int_0^t\langle g(s),L\varphi\rangle \dd s +\frac{\lambda^2}{2}\sum\limits_{j\geq 0} \int_0^t \left\langle  g(s),\left(F_j\cdot \left(\nabla_v+\dfrac{v}{2}\right)\right)^2\varphi\right\rangle \dd s, \quad \text{a.s.}
\end{array}
\end{equation}
\end{itemize}
Moreover for this solution the quantity $\rinf(g):=\iint g\dM \dd x\dd v$ is a.s. constant in time.
\end{theorem}
\noindent As follows from the classical properties of the Fokker-Planck operator, condition $(ii)$ may be equivalently written as : for any $T>0$, both $\nabla_vg\in L^2(\Omega\times(0,T)\times\T^N\times\R^N)$ and $v\,g\in L^2(\Omega\times(0,T)\times\T^N\times\R^N)$.

\medskip

As announced the solutions built in Theorem~\ref{thmex} provide better localization properties of $f$ with respect to $v$ but it is subjected to a restriction on the size of $\lambda$ in contrast with Theorem~\ref{thmexFP}. We will use the extra localization property to obtain convergence at exponential rate to an invariant measure. The constraint $|\lambda|<1$ arises to ensure that the localization property is not altered by the stochastic force term of Equation~\eqref{FP}. See also Remark~\ref{rk:lambda1} on the impact of the size of $\lambda$ on localization. 

As it is an important point in our analysis, let us give some technical insight on the role of $\lambda$ in proving existence for \eqref{eqg}. To ensure existence, we need that the random perturbation does not affect too much the dissipation of the operator $L$ so that the equation does not become effectively anti-diffusive. With this respect we emphasize that the singular force term in Equations~\eqref{FPVlasov} or \eqref{eqg} give rise to a second order differential operator with respect to $v$ when written in It\={o} form. For the original problem \eqref{FP}, this second order term has a good structure and is dissipative in $L^2$, \textit{cf.} Equation~\eqref{FPITO}. This accounts for the estimate \eqref{EEstimateFPsolutions} and the fact that there is no restriction on $\lambda$ in Theorem~\ref{thmexFP}. For Equation~\eqref{eqg} however, once put in It\={o} form as in \eqref{FPBisIto}, it appears that the corrective second order term has a structure that is not compatible even with basic energy estimates.

We obtain the existence of solutions to Equation \eqref{eqg} through a stochastic version of standard Galerkin schemes. Precisely, we project Equation \eqref{eqg} on some finite dimensional space. Doing so, we construct a sequence $(g_m)_m$ of approximate solutions to our problem. Then, one has to derive energy estimates on the sequence $(g_m)_m$ in order to take limits in the approximate projected problem. Another implementation of this strategy would likely prove Theorem~\ref{thmexFP}. Yet for comparison we provide a proof of Theorem~\ref{thmexFP} in Appendix~\ref{app:thmexFP} through a regularization of the mild formulation of \eqref{FP}, an approach that is probably even more classical for stochastic partial differential equations.

\bigskip

\paragraph{\textbf{Large-time behavior}}

Our third main result --- and the one that actually motivates our whole analysis --- is about existence, uniqueness and mixing properties of an invariant measure to problem \eqref{eqg}.

\begin{theorem}\label{thminv}
Suppose that hypothesis \eqref{noise} is satisfied. Let $\bar{\rho}\in\R$ and introduce the space
$$
X_{\bar{\rho}}\ :=\ \left\{ g\in L^2(\T^N\times\R^N)\,;\, \<g,\dM\> = \bar{\rho}\right\}\,.
$$
Then there exists $\lambda_0\in(0,1)$ independent on $\bar{\rho}$ such that, for $|\lambda|<\lambda_0$, the problem 
\begin{equation}\label{probm}\tag{$\text{P}_{\bar{\rho}}$}
\left\{
\begin{array}{ll}
& \dd g\ +\ v\cdot\nabla_x g\ \dd t\ +\ \lambda\left(\nabla_v-\dfrac{v}{2}\right)g\ \odot\dd W_t\ =\ Lg\ \dd t\\ [0.5em]
& g(0)=g_{\text{in}}\in X_{\bar{\rho}}
\end{array} \right.
\end{equation}
admits a unique invariant measure $\mu_{\bar{\rho}}$ on $X_{\bar{\rho}}$. Besides, there exist some constants $C\geq 0$, $\kappa>0$ depending on $\bar{\rho}$ and $N$ only, such that
\begin{equation}\label{mixingmu}
\left|\E\Psi(g(t))-\<\Psi,\mu_{\bar{\rho}}\>\right|\leq C e^{-\kappa t}\|g_{\text{in}}\|_{L^2(\T^N\times\R^N)},
\end{equation}
for every $\Psi\colon X_{\bar{\rho}}\to\R$ which is $1$-Lipschitz continuous.
\end{theorem}
Estimate \eqref{mixingmu} gives exponential convergence to the invariant measure $\mu_{\bar{\rho}}$ in the $1$-Was\-ser\-stein distance. Indeed, if $\mathcal{P}_1(X_{\bar{\rho}})$ is the set of Borel probability measures $\nu$ on $X_{\bar{\rho}}$ having finite first moment, that is, such that
$$
\int_{X_{\bar{\rho}}}\|g\|_{L^2(\T^N\times\R^N)}d\nu(g)<+\infty,
$$
then, thanks to the Kantorovitch duality theorem \cite[Theorem~5.10]{VillaniOldNew}, estimate \eqref{mixingmu} reads
$$
W_1(\mu(t),\mu_{\bar{\rho}})\leq C e^{-\kappa t}\|g_{\text{in}}\|_{L^2(\T^N\times\R^N)},
$$
where $\mu(t)$ is the law of $g(t)$ and $W_1$ is the $1$-Wasserstein distance on $\mathcal{P}_1(X_{\bar{\rho}})$:
$$
W_1(\mu,\nu)=\inf\left\{\iint_{X_{\bar{\rho}}\times X_{\bar{\rho}}}\|f-g\|_{L^2(\T^N\times\R^N)} d\pi(f,g)\right\},
$$
where the infimum is with respect to probability measures $\pi$ on $X_{\bar{\rho}}\times X_{\bar{\rho}}$ having first and second marginals $\mu$ and $\nu$ respectively. \bigskip

Some possibly growing-in-time uniform energy estimates are sufficient to prove existence and uniqueness of solutions to \eqref{eqg}. However to prove existence and uniqueness of an invariant measure for problem \eqref{eqg} we prove and use as our main tool adapted hypocoercive (and hypoelliptic) estimates. Therefore let us say a few words about the theory of hypocoercivity\footnote{See also Appendix~\ref{app:ellipticity} for an example of classical coercivity and ellipticity.}, as coined by Villani in \cite{villani}, in a simple context. It is particularly well-suited to providing rates of convergence towards equilibrium of solutions to kinetic collisional models. For instance, consider the following class of kinetic models
\begin{equation}\label{modelhypo}
\partial_tf + v\cdot\nabla_xf = Qf,
\end{equation}
where $Q$ is a linear collisional operator which acts on the velocity variable only, and choose some weighted-$L^2$ space $H_v$ such that $Q$ is symmetric on $L^2_x\otimes H_v$.  Also suppose that, denoting by $\Pl$ the orthogonal projection on $\textrm{ker}(Q)$, the following (local-in-space) weak coercivity assumption holds
$$\<Qh,h\> \leq -c \|h - \Pl h\|^2$$
for some $c>0$. This implies that $Q$ has a spectral gap when considered as acting on $H_v$, that is on functions homogeneous in space. The class of operators we have just introduced includes, among others, the cases of linearized Boltzmann, classical relaxation, Landau and Fokker-Planck equations. Note that while the global steady states of these models do belong to $\text{ker}(Q)$, the foregoing kernel is not reduced to Maxwellians so that the above weak coercivity fails to yield convergence to equilibrium. Introducing the global projection $\bar\Pi$ on $\text{ker}(-v\cdot\nabla_x+Q)$ defined by
$$
\bar\Pi h=\int_{\T^N} \Pl h(x,\cdot)\dd x.
$$
we first remark that, if $f$ is a solution to Equation \eqref{modelhypo}, $\bar\Pi f(t) = \bar\Pi f(0)$ is independent of time. Then the piece of information that stems from hypocoercivity theory is the exponential damping of the solution $f$ to equilibrium $\bar\Pi f(0)$
$$\|f(t)-\bar\Pi f(0)\|_{\mathcal{H}}\leq Ke^{-\tau t},\quad t\geq 0,$$
in some Sobolev space $\mathcal{H}$ built on $L^2_x\otimes H_v$. The key-point is that (local-in-velocity) weak coercivity estimates afforded by commutators of $-v\cdot\nabla_x$ and $Q$ may be incorporated in an energy estimate so as to ensure a full control of $\|f(t)-\bar\Pi f(t)\|_{\mathcal{H}}$. We refer the reader to the memoir of Villani \cite{villani} and references therein and also to the paper of Mouhot and Neumann \cite{mouhotneumann} that studies the convergence to equilibrium for many kinetic models including Fokker-Planck equations. Our approach to hypoellipticity is global and mimic hypocoercive estimates as in \cite{villani}. In the case of the deterministic Fokker-Planck equation \eqref{eqg} where $\lambda = 0$, the kernel of $-v\cdot\nabla_x+L$ is spanned by the function $\cM^{\frac{1}{2}}$ and
$$\bar\Pi g = \rho_{\infty}(g) \cM^{\frac{1}{2}},$$
where $\rho_{\infty}(g):=\iint g(t)\cM^{\frac{1}{2}}\,\dd x\,\dd v = \iint g(0)\cM^{\frac{1}{2}}\,\dd x\,\dd v$ (this quantity being time independent). And one can prove (see \cite[Section 5.3]{mouhotneumann}) an exponential damping for the quantity $g(t) - \rho_{\infty}(g) \cM^{\frac{1}{2}}$ in a weighted $H^1(\T^N\times\R^N)$ norm. \medskip

\noindent In the present paper, we establish hypocoercive estimates on the Fokker-Planck model \eqref{eqg}, that involves a perturbation by a random force. To handle the stochastic term we incorporate the corrections coming from the It\={o} formula in the roadmap of the proof of Mouhot and Neumann \cite{mouhotneumann}. By doing so we achieve 
\begin{equation}\label{HypoIntro}
\E\|g(t)\|_{L^2_{\nabla,D}}^2  \leq\ Ce^{-ct}\E\|g_\mathrm{in}\|_{L^2_{\nabla,D}}^2 + K \E|\rho_\infty(g)|^2, \quad t\geq 0,
\end{equation}
where $L^2_{\nabla,D}$ is a suitable weighted version of $H^1(\T^N\times \R^N)$ Sobolev space (see below \eqref{gradspaces} for the precise definition). 
In particular, any two solutions of the problem \eqref{eqg} $g_1$ and $g_2$ with respective initial conditions $g_{\text{in},1}$ and $g_{\text{in},2}$ such that $\rho_\infty(g_{\text{in},1})=\rho_\infty(g_{\text{in},1})$ meet exponentially fast in infinite time. A priori the latter is only guaranteed when $g_{\text{in}}$ belongs to $L^2_{\nabla,D}$. However following the same lines we also prove a hypoelliptic regularizing effect showing that the flow instantaneously sends $L^2$ to $L^2_{\nabla,D}$. \medskip

\noindent The existence of an invariant measure for problem~\eqref{eqg} follows then almost readily from hypoellipticity and hypocoercivity. Indeed existence is obtained from compactness\footnote{See Appendix~\ref{app:invariant}.} of time-averages that stems from compact embedding of $L^2_{\nabla,D}$ in $L^2_{x,v}$, instantaneous regularization and uniform-in-time bounds in $L^2_{\nabla,D}$. Concerning uniqueness and mixing, they stem from exponential convergence of stochastic trajectories.\medskip

\paragraph{\textbf{Further comments.}} Though hypocoercive and hypoelliptic estimates play a crucial role in our analysis, we caution the reader that the constraint on $|\lambda|$ is more intimately tied to issues of localization in the velocity variable. As we already pointed out, even for the Fokker-Planck operator acting on functions constant in space --- an elliptic operator --- some localization is needed to obtain global-in-time estimates. Yet choosing a framework ensuring localization necessarily breaks the skew-symmetric structure of the Stratonovich transport term, resulting in quasilinear contributions to energy estimates imposing some constraint on $|\lambda|$ even for the existence part of the argument.

Control on velocity spreading is indeed an ubiquitous issue in the analysis of (non relativistic) kinetic models. To some extent, for most problems concerning collisionless models this is actually the key issue. Collisional mechanisms, as encoded here by the Fokker-Planck operator, usually damp large velocities thus offering better localization. Yet as is readily apparent on characteristic equations~\eqref{XVsto} forcing terms may counterbalance collisional effects and lead to a loss of localization in large time. The constraint on noise intensity in Theorems~\ref{thmex} and~\ref{thminv} precisely enforces that collisions are the dominant mechanism acting on velocity localization.\medskip

\paragraph{\textbf{Plan of the paper.}} The proof of Theorem~\ref{thmexFP} is provided in Appendix~\ref{app:thmexFP}. In Section~\ref{sec:existsuniqueg} we introduce our framework with more precision and prove Theorem~\ref{thmex}. Hypoelliptic and hypocoercive estimates on the solution $g$ to \eqref{eqg} are proved in Section~\ref{sec:hypopo}, see Theorem~\ref{thmhypo}. They are applied in Section~\ref{sec:invmeas} to the proof of Theorem~\ref{thminv}. In Appendix~\ref{app:probaVFP} we recall the classical probabilistic interpretation of the Vlasov-Fokker-Planck operator and use it to prove some well-known estimates used to prove Theorem~\ref{thmexFP} and provide an explicit example, already summarized in Proposition~\ref{explicitmu}, showing that some restriction on $|\lambda|$ is indeed required in the framework of Theorems~\ref{thmex} and~\ref{thminv}. In Appendix~\ref{app:background} we also gather some basic background material that may be skipped by the expert reader but may be useful to the reader unfamiliar with some of the crucial underlying concepts.

\section{Existence and uniqueness of solutions}\label{sec:existsuniqueg}

\subsection{Preliminaries}

Prior to entering into the heart of our analysis we make precise our notational convention and recall some well-known facts concerning $L$.\smallskip

\noindent \textbf{It\={o} form.} To study the Cauchy problem \eqref{eqg}, we will work on its It\={o} form
\begin{equation}\label{FPBisIto}
\begin{array}{rcl}
\displaystyle \dd g + v\cdot\nabla_x g\ \dd t&+&\displaystyle
\lambda \left(\nabla_v-\dfrac{v}{2}\right)g \cdot \dd W_t - Lg\ \dd t\\[0.5em]
&=&\displaystyle
\frac{ \lambda^2}{2}\sum\limits_jF_j\cdot\left(\nabla_v-\dfrac{v}{2}\right)\left(F_j\cdot\left(\nabla_v-\dfrac{v}{2}\right)g\right)\dd t
\end{array}
\end{equation}
where we recall that 
$$
Lg\ =\ \Delta_v g\ +\ \left(\dfrac{N}{2}-\dfrac{|v|^2}{4}\right)g
$$
and that we always assume~\eqref{noise}.

To derive \eqref{FPBisIto} from \eqref{eqg}, we use the following computational rule. Let $h$ and $A$ be semi-martingales given by
$$
\begin{array}{rcl}
\displaystyle
\dd h&=&\displaystyle
\cF(h)\ \dd t\ +\ \cG(h)\odot\dd W\,,\\
\displaystyle
\dd A&=&\displaystyle
\cA(h)\ \dd t\ +\ \cB(h)\odot\dd W\,.
\end{array}
$$
Then we have\footnote{See also Appendix~\ref{app:ItoInt}.} \cite[Section~20.4]{SchillingPartzsch2014}
$$
\dd A\ =\ \cF(h)\ \dd t\ +\ \dfrac12\sum_{j\geq 0}F_j\cdot d\cB(h)(F_j\cdot\cG(h))\ \dd t\ +\ \cB(h)\cdot\dd W.
$$
Applied to our original formulation, this gives the following I\={o} form of \eqref{FP}
\begin{equation}\label{FPITO}
\dd f\ +\ v\cdot\nabla_x f\ \dd t\ +\ \lambda \nabla_vf\cdot\dd W_t\ =\ \cQ(f)\ \dd t+\sum_{j\geq 0}(F_j\cdot\nabla_v)^2 f
\end{equation}
where $\cQ$ is given by \eqref{defQQ}.

\noindent \textbf{Functional Spaces.} In the following, we denote by $\langle\cdot,\cdot\rangle$ and $\|\cdot\|$ respectively the scalar product and the norm of $L^2_{x,v}:=L^2(\T^N\times\R^N)$. For any Hilbert space $H$ and any $T>0$, we denote by $\mathcal{C}_w([0,T];H)$ the space of functions on $[0,T]$ with values in $H$ that are continuous for the weak topology of $H$. Let us introduce the differential operators
$$
D=\nabla_v+\frac{v}{2},\quad D^*=-\nabla_v+\frac{v}{^2},
$$
where $D^*$ is the formal adjoint of $D$ component-wise, \textit{i.e.} $D^*_k=(D_k)^*$, $k=1,...,N$. Note that, for $f$ sufficiently smooth and localized,
\begin{equation}
\label{Df}
\left\|Df\right\|^2\ =\
\left\|\nabla_vf\right\|^2+\frac{1}{4}\left\|vf\right\|^2\ -\ \dfrac{N}{2}\|f\|^2
\end{equation} 
and
\begin{equation}
\label{Dstarf}
\left\|D^*f\right\|^2\ =\
\left\|\nabla_vf\right\|^2+\frac{1}{4}\left\|vf\right\|^2\ +\ \dfrac{N}{2}\|f\|^2.
\end{equation} 
We introduce the space 
$$
L^2_D=\{f\in L^2(\R^N);\,Df\in L^2(\R^N)\}=\{f\in L^2(\R^N);\,D^*f\in L^2(\R^N)\}
$$
and then define the spaces 
\begin{equation}\label{gradspaces}
L^2_{x,D}=L^2(\T^N;L^2_D)\,,\quad L^2_{\nabla,D}=\{f\in L^2_{x,D}\,;\quad\nabla_x f\in L^2_{x,v}\},
\end{equation}
equipped respectively with norms
$$
\|f\|^2_{L^2_{x,D}}=\|D^*f\|^2,\quad \|f\|^2_{L^2_{\nabla,D}}=\|D^*f\|^2+\|\nabla_x f\|^2.
$$
\textbf{Fokker-Planck Operator.} We introduce the transport operator $A=  v\cdot\nabla_x$ which is skew-adjoint, that is which satisfies $A^*=-A$. Concerning the Fokker-Planck operator $L$, we gather hereafter some of its properties. Since we have derived it from $\cQ$, we abuse terminology and call $L$ itself a Fokker-Planck operator. However, as already mentioned, $L$ is the quantum harmonic oscillator, and operators $D^*$ and $D$ below are associated creation and annihilation operators. For this reason, the properties recalled below are well-known in the mathematical physics literature and may be found for instance in \cite[Appendix to V.3]{ReedSimonI} or \cite[Section~1.3]{Helffer}.

First, we recall the expression
$$
Lf\ =\ \Delta_v f\ +\ \left(\dfrac{N}{2}-\dfrac{|v|^2}{4}\right)f.
$$
Alternatively $L$ is also given by
\begin{equation}\label{LDstarD}
Lf\ =\ -\sum_k D_k^*D_kf
\ =\ Nf-\sum_k D_kD_k^*f,
\end{equation}
which we will denote $L = -D^*D = N\text{Id} - DD^*$ for short. From \eqref{LDstarD} follows immediately the dissipative bound
\begin{equation}
\label{fLf}
-\left\langle f,Lf\right\rangle\ =\ \left\|Df\right\|^2.
\end{equation}
Recall that the operator $\cQ$ is defined by \eqref{defQQ}. To factor out Gaussians from eigenfunctions, since $\cQ$ has divergence structure we temporarily consider the formal adjoint of $\cQ$ on $L^2(\R^N)$, that is
\begin{equation}\label{AdjointQ}
\cQ^*\colon f\mapsto \Delta_v f- v\cdot\nabla_v f.
\end{equation}
The operator $\cQ^*$ is self-adjoint on $L^2(\R^N,\gamma)$, where $\gamma$ is the Gaussian measure with density $\cM$ with respect to the Lebesgue measure on $\R^N$. For $j\in\N^N$, Hermite polynomials
\begin{equation}\label{defHj}
H_j(v)=\frac{(-1)^{|j|}}{\sqrt{j!}}\cM^{-1}\partial_v^j(\cM),
\end{equation}
where
$$
|j|=j_1+\ldots+j_N,\quad j!=j_1!\cdots j_N!,\quad \partial_v^j=\partial_{v_1}^{j_1}\cdots\partial_{v_N}^{j_N},
$$
form an orthonormal basis of $L^2(\R^N,\gamma)$ of eigenvectors of $\cQ^*$ 
$$
\cQ^*H_j=-|j|H_j.
$$
The operator $L$ is related to the operator $\cQ^*$ by the formula $Lf=\dM\cQ^*(\cM^{-\frac12}f)$. It follows that, setting $q_j=\dM H_j$, we obtain a Hilbert basis of $L^2(\R^N)$ constituted of eigenvectors of $L$ associated with eigenvalues $-|j|$. There is a compact expression of $q_j$: using the formula $\partial_v^j(\dM f)=(-1)^{|j|}\cM^{\frac12} \left[D^*\right]^j f$ (which can be proved by recursion on $|j|$), and the definition \eqref{defHj} of Hermite polynomials, we obtain
\begin{equation}\label{exprqj}
q_j=\frac{1}{\sqrt{j!}}\left[D^*\right]^j \dM.
\end{equation}
The formula \eqref{exprqj} gives in particular the first identity in the following formulas
\begin{equation}\label{translationqj}
D^*_k q_j=\sqrt{j_k+1}\ q_{j+e_k},\qquad D_k q_j=\sqrt{j_k}\ \mathbf{1}_{j_k>0}\ q_{j-e_k},
\end{equation}
for $j\in\N^N$, $k\in\{1,\ldots,N\}$. The formula for $D_k q_j$ is obtained by duality, for instance by computing the $l$-th coefficients $\<D_k q_j,q_l\>=\<q_j,D_k^* q_l\>$.\medskip

\noindent \textbf{Eigenspaces.} Let $(p_k)_{k\in\Z^N}$ denote the standard trigonometric Hilbert basis of $L^2(\T^N)$, explicitly given by $p_k(x)=e^{2\pi i k\cdot x}$. In particular it is formed by normalized eigenfunctions for the Laplacian $-\Delta_x$. Remember that $(q_j)_{j\in\N^N}$ is the spectral Hilbert basis for the Fokker-Planck operator $L$ on $L^2(\R^N)$ introduced above. We define the Hilbert basis $(e_{k,l})_{(k,l)\in \Z^N\times\N^N}$ of $L^2_{x,v}$ by 
$$e_{k,l}(x,v):=p_k\otimes q_l (x,v)=p_k(x)q_l(v),\quad (k,l)\in \Z^N\times\N^N,\,x\in\T^N,\,v\in\R^N.$$

For any $(k_0,l_0)\in (\N\cup\{\infty\})^2$, we set 
$$
E_{k_0,l_0}\ :=\ \mathrm{Closure}_{L^2_{x,v}}\big(\mathrm{Span}\ \{\ e_{k,l}\ ; |k|\leq k_0\ \textrm{and}\ |l|\leq l_0\ \}\big)\ 
$$ 
and introduce $\Pi_{k_0,l_0}$ the $L^2_{x,v}$ orthogonal projection on $E_{k_0,l_0}$. When $k_0=l_0$, we simplify notation to $E_{k_0}$ and $\Pi_{k_0}$. In particular $\bar\Pi=\Pi_{0,0}$ and $\Id=\Pi_{\infty,\infty}$. By \eqref{translationqj}, we have the commutation rules
\begin{equation}\label{DPI}
\Pi_m D^*=D^*\Pi_{m,m-1},\quad D\Pi_m=\Pi_{m,m-1}D
\end{equation}
for all $m\geq 1$. These identities will be used to derive hypocoercive estimates on the ap\-proxi\-ma\-te Galerkin solution to \eqref{eqg}.\medskip

\noindent Let us also introduce the orthogonal projector $\Pl=\Pi_{\infty,0}$ on $L_x^2\otimes\mathrm{Span}\{q_0\}$
$$
\Pl(f)(x,v)\ =\ \langle\dM,f(x,\cdot)\rangle_{L^2_v(\R^N)}\ \dM(v),\qquad \Pl^\bot=I-\Pl.
$$
Then, we have
\begin{equation}\label{Ldissip}
-\left\langle f,Lf\right\rangle\geq \|\Pl^\bot f\|^2.
\end{equation}
Using \refe{fLf}, we then deduce from \refe{Ldissip} that
\begin{equation}\label{fDpi}
\|f\|^2\leq \|\Pl f\|^2+\|Df\|^2.
\end{equation}
Finally, in the sequel, we denote by $\{T,T'\}:=TT'-T'T$ the commutator of two operators $T$ and $T'$. We point out that one readily shows the following algebraic identities
$$
\{D,A\} = \nabla_x, \qquad \{D,D^*\} = N\text{Id},
$$
and stress that the former identity is the cornerstone of both our hypocoercive and hypoelliptic estimates.

\subsection{The Galerkin scheme}

We are now ready to prove the existence part of Theorem~\ref{thmex}. To do so, we use a Galerkin projection method. Here we project Equation \eqref{eqg} onto the finite dimensional space $E_m$ and seek a solution of the projected equation valued in this finite-dimensional subspace, then we take the limit $m\to\infty$ when this finite subset $E_m$ increases up to the whole Hilbert space $L^2$. 

To start with we prove the existence of an approximate solution $g_m : [0,T]\times\Omega\to E_m$ to a projected version of $\eqref{eqg}$ in the following result.

\begin{proposition}\label{p:exGalerkin}
Suppose that hypothesis $(\ref{noise})$ holds and let $g_\mathrm{in}\in L^2(\Omega;L^2_{x,v})$. 
For all $m\geq 0$ and any $T>0$, there exists a unique adapted process $g_m\in \mathcal{C}([0,T];L^2(\Omega;E_m))$ satisfying, for all $t\in [0,T]$, for all $\varphi\in E_m$,
\begin{equation}\label{Galerkin}
\begin{array}{lll}
& \displaystyle \langle g_m(t),\varphi\rangle  &=\ \displaystyle \langle \Pi_m g_\mathrm{in},\varphi\rangle  + \int_0^t \langle g_m(s), v\cdot\nabla_x \varphi\rangle \dd s + \lambda\sum_{j\geq 0}\int_0^t \left\langle g_m(s),F_j\cdot D \varphi\right\rangle \dd\beta_j(s) \\
&& \displaystyle+\int_0^t\langle g_m(s),L\varphi\rangle \dd s +\frac{\lambda^2}{2}\sum\limits_{j\geq 0} \int_0^t \left\langle  g_m(s),\left(F_j\cdot D\right)^2\varphi\right\rangle \dd s, \quad \text{a.s.}
\end{array}
\end{equation}
Moreover, if $|\lambda|< 1$, then
\begin{equation}\label{uniform-Galerkin}
\frac12\max_{t\in[0,T]} e^{-2N\lambda^2\,t}\E\|g_m(t)\|^2\ +\ (1-\lambda^2)\ \int_0^T e^{-2N\lambda^2\,t}\E\|D\,g_m(t)\|^2\dd t
\ \leq\ \frac12\E\|g_\mathrm{in}\|^2\,.
\end{equation}
\end{proposition}

\textbf{Proof of Proposition~\ref{p:exGalerkin}.} For $g_m\in \mathcal{C}([0,T];L^2(\Omega;E_m))$, Equations \eqref{Galerkin} are equivalently written --- in terms of the coefficients $d_{k,l}=\langle g_m,e_{k,l}\rangle$, $|k|\leq m$ and $|l|\leq m$, of $g_m$ --- as a finite-dimensional It\={o} system with globally Lipschitz coefficients (as it is linear and finite-dimensional). It follows then from standard arguments that there exists a unique adapted and continuous process $g_m\in \mathcal{C}([0,T];L^2(\Omega;E_m))$ satisfying \eqref{Galerkin}.  Assume now $|\lambda|<1$. To derive the uniform bound \eqref{uniform-Galerkin}, we multiply \eqref{Galerkin} by $d_{k,l}$ and sum over $k$ and $l$ to obtain  
\begin{align}
\dfrac12\dfrac{\dd}{\dd t}\E\|g_m\|^2+\E\|D g_m\|^2&=\displaystyle
\lambda^2\,\dfrac12\E\sum_{j} \< (F_j\cdot D^*)^2g_m,g_m\>+\< (F_j\cdot D^*)g,(F_j\cdot D^*)g_m\>\nonumber\\
&=\displaystyle
\lambda^2\,\E\sum_{j} \< (F_j\cdot D^*)g_m,(F_j\cdot \frac{D+D^*}{2})g_m\>\nonumber\\
&\leq \lambda^2\,\E\|D^*g_m\|^2\ \leq\ \lambda^2\,\E\|Dg_m\|^2+N\lambda^2\E\|g_m\|^2,\label{EEmm}
\end{align}
from which the bound follows since \eqref{Galerkin} implies $g_m(0)=\Pi_m g_\mathrm{in}$ and therefore $\E\|g_m(0)\|^2\leq\E\|g_\mathrm{in}(0)\|^2$. Here above we have used both $\|Df\|\leq \|D^*f\|$ and $\|D^*f\|^2=\|Df\|^2+N\|f\|^2$.
\qed\bigskip

\noindent Obviously, alternatively we may view $g_m$ as belonging to $\mathcal{C}([0,T];L^2(\Omega;L^2_{x,v}))$ and satisfying
\begin{equation}\label{GalerkinBis}
\dd g_m + \Pi_m(v\cdot\nabla_x g_m)\dd t -\lambda \Pi_m (D^* g_m \odot\dd W_t) = L g_m \dd t,
\end{equation}
with initial condition
$$
g_m(0) = \Pi_mg_{\text{in}}.
$$
This does imply that for any $t$, a.s. $\Pi_m g_m(t)=g_m(t)$ hence $g_m(t)\in E_m$.

\subsection{Proof of Theorem \ref{thmex}}\label{sec:exuniq}
\noindent In this section, we prove Theorem \ref{thmex}. \medskip

{\em Limit point.} Let $T>0$. We use estimate \eqref{uniform-Galerkin} to obtain uniform bounds on $g_m$ in $L^{\infty}(0,T;L^2(\Omega;L^2_{x,v}))$ and on $Dg_m$ in $L^2(\Omega\times(0,T);L^2_{x,v})$ by some quantities depending on $N$, $T$, $\lambda$ and the norm $\E\|g_{\text{in}}\|^2$. As a consequence, $(g_m)_m$ admits a subsequence (still denoted $(g_m)_m$ for simplicity) such that 
$$
g_m \rightharpoonup g\mbox{ in }L^2(\Omega\times(0,T);L^2_{x,v})
$$
where $g,\,Dg\in L^2(\Omega\times(0,T);L^2_{x,v})$. \medskip

{\em Time continuity.} 
We need to upgrade the foregoing convergence to convergence in the space $L^2(\Omega;\mathcal{C}_w([0,T];L^2_{x,v}))$. This follows from the fact that the equation provides uniform bounds in $L^2(\Omega;\mathcal{C}^\alpha([0,T];X^*))$ for some $\alpha>0$ and some space $X$ continuously embedded and dense in $L^2_{x,v}$, that may be combined with uniform bounds in $L^2(\Omega;\mathcal{C}([0,T];L^2_{x,v}))$.

To be explicit, let $X$ denote the Banach space of elements $\varphi\in L^2_{x,v}$ with finite norm
$$
\|\varphi\|_X=\left[\iint_{\T^N\times\R^N} (1+|v|^2)\left(|\varphi(x,v)|^2+|\nabla_{x,v}\varphi(x,v)|^2+|\Delta_v\varphi(x,v)|^2\right) \dd x \dd v\right]^{1/2}.
$$
By \eqref{EEmm}, we have 
\begin{equation}\label{supEgm}
\sup_{t\in[0,T]}\E\|g_m(t)\|^2\leq e^{2N\lambda^2 T}\E\|\Pi_m g_\mathrm{in}\|^2=:C_T\E\|\Pi_m g_\mathrm{in}\|^2.
\end{equation}
This implies the following series of estimates. First
\begin{align}
\E\left|\int_\sigma^t \langle g_m(s), v\cdot\nabla_x \varphi\rangle \dd s\right|^2
&\leq |t-\sigma|^2\sup_{s\in[0,T]}\E\|g_m(s)\|^2\|\varphi\|_X^2\nonumber\\
&\leq |t-\sigma|^2 C_T\E\|\Pi_m g_\mathrm{in}\|^2\|\varphi\|_X^2.\label{Cw1}
\end{align}
Similarly, we have
\begin{equation}\label{Cw2}
\E\left|\int_\sigma^t \langle g_m(s),L\varphi\rangle \dd s\right|^2\leq |t-\sigma|^2 C_T\E\|\Pi_m g_\mathrm{in}\|^2\|\varphi\|_X^2,
\end{equation}
and, using the hypothesis \eqref{noise},
\begin{equation}\label{Cw3}
\E\left|\int_\sigma^t \sum\limits_{j\geq 0} \left\langle  g_m(s),\left(F_j\cdot D\right)^2\varphi\right\rangle \dd s\right|^2\leq |t-\sigma|^2 C_T\E\|\Pi_m g_\mathrm{in}\|^2\|\varphi\|_X^2.
\end{equation}
By the Burkholder-Davis-Gundy Inequality, \cite[Lemma~7.7]{daprato}, and \eqref{noise}, we can also estimate as follows the time increments in the stochastic integral
\begin{align}
\E\left|\int_\sigma^t  \sum_{j\geq 0} \left\langle g_m(s),F_j\cdot D \varphi\right\rangle \dd\beta_j(s)\dd s\right|^4
&\leq C_{\mathrm{BDG}}\E\left|\sum_{j\geq 0} \int_\sigma^t  \left| \left\langle g_m(s),F_j\cdot D \varphi\right\rangle\right|^2 \dd s\right|^2\nonumber\\
&\leq C_{\mathrm{BDG}}|t-\sigma|^2\sup_{s\in[0,T]}\E\|g_m(s)\|^4\|\varphi\|_X^4\nonumber\\
&\leq C_{\mathrm{BDG}}|t-\sigma|^2 C_T^2\left[\E\|\Pi_m g_\mathrm{in}\|^2\right]^2\|\varphi\|_X^4.\label{Cw4}
\end{align}
By \eqref{Cw1}, \eqref{Cw2}, \eqref{Cw3}, \eqref{Cw4} and the equation \eqref{Galerkin}, we obtain 
\begin{equation}\label{Cw45}
\E\left|\langle g_m(t),\varphi\rangle-\langle g_m(\sigma),\varphi\rangle\right|^4\leq |t-\sigma|^{2} \tilde{C}_T\left[\E\|\Pi_m g_\mathrm{in}\|^2\right]^2\|\varphi\|_X^4,
\end{equation}
for all $\varphi\in X$, for all $t,\sigma\in[0,T]$, where $\tilde{C}_T$ is a constant depending on $N$, $\lambda$, $T$ only. This implies 
\begin{equation}\label{Cw5}
\E\left|\langle g(t),\varphi\rangle-\langle g(\sigma),\varphi\rangle\right|^4\leq |t-\sigma|^{2} \tilde{C}_T\left[\E\|g_\mathrm{in}\|^2\right]^2\|\varphi\|_X^4,
\end{equation}
for all $\varphi\in X$, for all $t,\sigma\in[0,T]$. Let $\alpha\in(0,1/4)$. At fixed $\varphi\in X$, we deduce from \eqref{Cw5} and the Kolmogorov's Criterion, \cite[Theorem~3.4]{daprato} that $t\mapsto \langle g(t),\varphi\rangle$ has a modification in $L^2(\Omega;\mathcal{C}^\alpha([0,T]))$. More precisely, there exists $A_{t,\varphi}$ an event of probability $1$ and a process $G_\varphi(t)$ such that, for all $t\in[0,T]$, $\mapsto \langle g(t),\varphi\rangle=G_\varphi(t)$ on $A_{t,\varphi}$, and the process $G_\varphi$ satisfies
\begin{equation}\label{Cw6}
\E\left|\sup_{t\not=s\in[0,T]}\frac{|G_\varphi(t)-G_\varphi(s)|}{|t-s|^\alpha}\right|^2\leq \bar C_T\E\|g_\mathrm{in}\|^2\|\varphi\|_X^2,
\end{equation}
where $\bar C_T$ is a constant depending on $N$, $\lambda$, $T$, $\alpha$, $\E\|g_\mathrm{in}\|$ only.

We now use a density argument to go from strong $X$ to weak $L^2_{x,v}$. To do so we choose a metrization of the weak topology of $L^2_{x,v}$ on bounded sets. Let $\mathcal{D}$ be a dense and countable subset both of $X$ and $L^2_{x,v}$ with their own topology. For instance we could set $\mathcal{D}$ to be the set of finite linear combinations with coefficients in $\mathbb{Q}$ of vectors $e_{k,l}$, $k\in\Z^N$, $l\in\N^N$. Let $A_t=\cap_{\varphi\in\mathcal{D}}A_{t,\varphi}$. Since, almost-surely, $\varphi\mapsto G_\varphi(t)$ is linear with norm bounded by $\|g(t)\|$, there exists a process $(\tilde g(t))$ on $L^2_{x,v}$ such that, for all $t\in[0,T]$, for all $\varphi\in\mathcal{D}$, $ \langle g(t),\varphi\rangle= \langle \tilde g(t),\varphi\rangle$ on $A_t$. This implies $\tilde g(t)=g(t)$ (identity in $L^2_{x,v}$) on $A_t$, \textit{i.e.} $\tilde g$ is a modification of $g$. We have then \eqref{Cw6}, with $\langle \tilde g(t),\varphi\rangle$ instead of $G_\varphi(t)$. To conclude, let $\mathcal{D}=\{\varphi_1,\varphi_2,\ldots\}$ be an enumeration of $\mathcal{D}$ and let us introduce the distance
$$
d(f,g)=\sum_{n=1}^\infty \frac{|\<f-g,\varphi_n\>|}{2^n(1+\|\varphi_n\|_X)}.
$$
Note that
\begin{equation}\label{distsquare}
|d(f,g)|^2\leq \sum_{n=1}^\infty \frac{|\<f-g,\varphi_n\>|^2}{2^n(1+\|\varphi_n\|_X)^2}.
\end{equation}
The distance $d$ gives a metric compatible with the weak topology on the balls of $L^2_{x,v}$. Besides,
\begin{equation}\label{defCwalpha}
\|g\|_{\mathcal{C}_w^\alpha([0,T];L^2_{x,v})}^2=\|g(0)\|_{L^2_{x,v}}^2+\left|\sup_{t\not=s\in[0,T]}\frac{d(g(t),g(s))}{|t-s|^\alpha}\right|^2
\end{equation}
defines a norm of Banach space on $\mathcal{C}_w^\alpha([0,T];L^2_{x,v})$. By \eqref{distsquare}, we have
\begin{equation}\label{appdistsquare}
\E\left|\sup_{t\not=s\in[0,T]}\frac{d(g(t),g(s))}{|t-s|^\alpha}\right|^2
\leq\sum_{n=1}^\infty\frac{1}{2^n(1+\|\varphi_n\|_X)^2}\E\left|\sup_{t\not=s\in[0,T]}\frac{|\<g(t)-g(s),\varphi_n\>|}{|t-s|^\alpha}\right|^2,
\end{equation}
which is bounded by $\bar C_T\E\|g_\mathrm{in}\|^2$ due to \eqref{Cw6} (applied with $\langle \tilde g(t),\varphi\rangle$ instead of $G_\varphi(t)$). We deduce the estimate 
\begin{equation}\label{CwHolder}
\E\|g\|_{\mathcal{C}_w^\alpha([0,T];L^2_{x,v})}^2\leq (1+\bar C_T)\E\|g_\mathrm{in}\|^2.
\end{equation}
In particular, we have $g\in L^2(\Omega;\mathcal{C}_w([0,T];L^2_{x,v}))$.\medskip

{\em Existence.} By linearity of the equation, we can apply our estimates to $g_n-g_m$ instead of $g_n$ and let $n\to+\infty$: this gives 
$$
\E\|g-g_m\|_{\mathcal{C}_w^\alpha([0,T];L^2_{x,v})}^2 \leq (1+\bar C_T)\E\|(\mathrm{Id}-\Pi_m)g_\mathrm{in}\|^2.
$$
In particular, $(g_m)$ is converging to $g$ in $L^2(\Omega;\mathcal{C}_w([0,T];L^2_{x,v}))$. We now have all in hands to take the limit $m\to\infty$ in \eqref{Galerkin}. We deduce the existence of a solution $g$ satisfying the points $(i)$, $(ii)$ and $(iii)$ of Theorem \ref{thmex}.\medskip

{\em Uniqueness.} If $g\in L^2(\Omega;\mathcal{C}_w([0,T];L^2_{x,v}))$ solves \eqref{eqg} in the sense of $(i)$, $(ii)$ and $(iii)$ of Theorem \ref{thmex}, then $g$ satisfies the energy estimate
\begin{equation}\label{energyunique}
\dfrac12\E \|g(t)\|^2+\int_0^t \E\|Dg(s)\|^2 \dd s \leq \lambda^2\int_0^t \E[\|Dg(s)\|^2+N\|g(s)\|^2] \dd s+\dfrac12\E \|g(0)\|^2,\quad t\geq 0.
\end{equation}
Since $|\lambda|<1$, \eqref{energyunique} immediately gives, with Gronwall's lemma, that a solution with initial condition $g_{\text{in}}\equiv 0$ is zero in $L^{\infty}(0,T;L^2(\Omega;L^2_{x,v}))$ for every $T>0$. Hence the uniqueness by linearity of the problem. 
To prove \eqref{energyunique}, on the basis of $(i)$, $(ii)$ and $(iii)$ of Theorem \ref{thmex}, we apply the weak formulation \eqref{Solution} with $\varphi=e_{k,l}$ and use the It\={o} Formula. Note that the differential of $g\mapsto |\<g,e_{k,l}\>|^2$ at $g$ is 
$$
f\mapsto 2\mathrm{Re}\Big(\overline{\<g,e_{k,l}\>} \<f,e_{k,l}\>\Big).
$$ 
Since $g(t)$ is real-valued and $e_{k,l}(x,v)=e^{-2\pi i k\cdot x} q_l(v)$ where $q_l(v)$ is real-valued, the term 
$$
\mathrm{Re}\Big(\overline{\<g(t),e_{k,l}\>} \<g(t),v\cdot\nabla_x e_{k,l}\>\Big)
$$
vanishes and we obtain
\begin{align*}
\dfrac12\E |\<g(t),e_{k,l}\>|^2=&\dfrac12\E |\<g(0),e_{k,l}\>|^2-\E\int_0^t |l| |\<g(s),e_{k,l}\>|^2 \dd s\\
&+\dfrac{\lambda^2}{2}\sum_{j\geq 0}\mathrm{Re}\ \E\int_0^t\Big[ \overline{\<g(s),e_{k,l}\>} \<g(s),(F_j\cdot D)^2 e_{k,l}\>+| \<g(s),(F_j\cdot D) e_{k,l}\>|^2\Big] \dd s.
\end{align*}
We sum the result over $k,l$ and use Properties $(i)-(ii)$ of Theorem \ref{thmex} and the Bessel Identity to obtain
\begin{align*}
\dfrac12\E\|g(t)\|^2
+&\int_0^t \E\|Dg(s)\|^2 \dd s\\ 
=&\dfrac12\E\|g(0)\|^2
+\lambda^2\sum_{j\geq 0} \E\int_0^t\<(F_j\cdot D^*)g(s),(F_j\cdot \tfrac{D+D^*}{2})g(s)\> \dd s.
\end{align*}
Estimate \eqref{energyunique} then follows from Hypothesis~\eqref{noise} on the size of the coefficients of the noise as in \eqref{EEmm}.\medskip

{\em Conservation.} The fact that the quantity $\rho_{\infty}(g)$ is constant in time follows from setting $(k,l)=(0,0)$ in \eqref{Galerkin} and taking the limit $m\to\infty$. \qed \medskip

\section{Regularization and decay}\label{sec:hypopo}

\noindent We prove now extra properties for solutions provided by Theorem~\ref{thmex} summarized in the following theorem. 

\begin{theorem}\label{thmhypo}
Suppose that hypothesis $(\ref{noise})$ holds.
There exists $0<\lambda_0(N)<1$ such that, for all $|\lambda|<\lambda_0$ and any $g_\mathrm{in}\in L^2(\Omega;L^2_{x,v})$, the solution $g$ given by Theorem~\ref{thmex} satisfies the following properties. 
\begin{enumerate}
\item[\textit{i.}] the solution $g$ gains regularity instantaneously: for any $t_0>0$, there exists a constant $C(N,t_0)>0$ such that
\begin{equation}\label{thmregul}
\E\|g(t_0)\|^2_{L^2_{\nabla,D}}\leq C(N,t_0)\E\|g_{\text{in}}\|^2.
\end{equation}
\item[\textit{ii.}] the function $g$ satisfies the bound
\begin{equation}\label{Hypolow}
\begin{array}{rcl}\displaystyle
\E\|g(t)\|_{L^2_{\nabla,D}}^2  \!\!\!
&+&\displaystyle
c\ \E\int_{t_0}^t (\|g(s)\|_{L^2_{\nabla,D}}^2 \!\!\! + \|D\nabla_x g(s)\|^2  + \|D^2 g(s)\|^2 )\,\dd s\\[1em]
&\leq&\displaystyle
C\E\|g(t_0)\|_{L^2_{\nabla,D}}^2 \!\!\!+C \E|\rho_\infty|^2 (t-t_0),
\end{array}
\end{equation}
and the hypocoercive estimate
\begin{equation}\label{Hypostrong}
\E\|g(t)\|_{L^2_{\nabla,D}}^2  \leq\ Ce^{-c(t-t_0)}\E\|g(t_0)\|_{L^2_{\nabla,D}}^2 + K \E|\rho_\infty(g)|^2
\end{equation}
for all $t\geq t_0$, where the constants $c$, $C$ and $K$ depend on $N$ only.
\end{enumerate}
\end{theorem}

\subsection{Termwise estimates} \label{sec:estimates}

\noindent 
In this subsection, we derive some estimates on various functionals of the approximate solutions $(g_m)_m$. Next, we shall combine these termwise bounds to deduce hypocoercive estimates as in \eqref{Hypostrong} (see Section \ref{sec:hypo}) and follow a similar strategy to obtain regularization properties through hypoelliptic estimates including \eqref{thmregul} (see Section \ref{sec:regul}). 

\subsubsection{Heuristics}

\noindent Our strategy involves an estimation of a linear combination of $\E \Phi(g)$ where $\Phi$ is a quadratic functional of the form
$$
\Phi(g)=\< Sg,Tg\>,
$$
where $S$ and $T$ are operators in the variables $x$ or $v$ of order at most one. In particular, $S$ and $T$ are \textit{linear}. The rigorous procedure that we follow hereafter is to bound $\E \Phi(g_m)$ and take the limit $m\to\infty$, since all computations are readily justified when applied to the finite-dimensional system satisfied by $g_m$. 

However, for exposition purpose, proceeding in a formal way, we first explain the spirit of our computations on Equation \refe{eqg} satisfied by $g$. Apply $S$ to \refe{eqg} and then test against $Tg$, and do the same with the roles of $S$ and $T$ exchanged, to obtain
\begin{multline}\label{ST0}
\dd\Phi(g)=-\<SAg,Tg\>\dd t
+\lambda\sum_{j} \<S (F_j\cdot D^*)g,Tg\>\circ \dd\beta_j(t) + \<SLg,Tg\>\dd t +\mathsf{sym},
\end{multline}
where by ``$B(S,T)+\mathsf{sym}$" in the right-hand side of \refe{ST0}, we mean $B(S,T)+B(T,S)$. Switching to It\={o} form and taking expectation in \refe{ST0} gives
\begin{align}
\dfrac{\dd}{\dd t}\E\Phi(g)&=-\E\<SAg,Tg\> + \E\<SLg,Tg\> + \frac{\lambda^2}{2}\E\,\mathcal{N}_{S,T}(g)+\mathsf{sym}.\label{ST1}
\end{align}
where we have introduced the piece of notation
$$\mathcal{N}_{S,T}(g):=\sum_{j} \<S (F_j\cdot D^*)^2g,Tg\>+\<S (F_j\cdot D^*)g,T(F_j\cdot D^*)g\>.$$
Note also, in the case $S=T$, that, by \refe{fLf},
\begin{equation}
\E\<SLg,Sg\>=\E\<LSg,Sg\>+\E\<\{S,L\}g,Sg\>
=-\E\|D Sg\|^2+\E\<\{S,L\}g,Sg\>,\label{Lterms}
\end{equation}
thus, modulo a commutator, the term $\E\<SLg,Sg\>$ in \refe{ST1} provides the part $-\E\|DSg\|^2$ whose contribution helps to set up our hypocoercive estimates. In contrast control on space derivatives is gained by examining the case $S=\nabla_x$, $T=D$ and noticing that
$$
-\E\<DAg,\nabla_xg\>\ =\ -\E\<\{D,A\}g,\nabla_xg\>\ -\ \E\<ADg,\nabla_xg\>\ =\ -\E\|\nabla_xg\|^2\ -\ \E\<ADg,\nabla_xg\>
$$
provides the missing $\E\|\nabla_xg\|^2$. 

To proceed with the actual proof we modify the definition of $\cN_{S,T}$ to
$$\mathcal{N}^{(m)}_{S,T}(g):=\sum_{j} \<S (\Pi_m (F_j\cdot D^*))^2g,Tg\>+\<S \Pi_m(F_j\cdot D^*)g,T\Pi_m(F_j\cdot D^*)g\>$$ 
so as to reflect the presence of a projector in \eqref{GalerkinBis}.

\subsubsection{First estimate: $\E\|g_m\|^2$}  
\noindent We have already showed along the proof of Proposition~\ref{p:exGalerkin} that by taking $S=T=\mathrm{Id}$, one obtains
\begin{equation}
\label{hypo1}
\dfrac12\dfrac{\dd}{\dd t}\E\|g_m\|^2+\E\|D g_m\|^2\leq \lambda^2\E\|D^*g_m\|^2\,.
\end{equation}

\subsubsection{Second estimate: $\E\|\nabla_xg_m\|^2$} 
\noindent By choosing $S=T=\nabla_x$, we obtain, due to the fact that $A$ is skew-symmetric,
$$
\displaystyle
\dfrac12\dfrac{\dd}{\dd t}\E\|\nabla_x g_m\|^2+\E\|D \nabla_x g_m\|^2 
\leq \frac{\lambda^2}{2}\E\,\mathcal{N}^{(m)}_{\nabla_x,\nabla_x}(g_m),
$$
where $\mathcal{N}^{(m)}_{\nabla_x,\nabla_x}(g_m)$ is also written as
$$
\begin{array}{rcl}\displaystyle
\frac12\mathcal{N}^{(m)}_{\nabla_x,\nabla_x}(g_m)&=&\displaystyle
\sum_{j} \< \Pi_m(F_j\cdot D^*)\nabla_xg_m,(F_j\cdot \tfrac{D+D^*}{2})\nabla_xg_m\>\\[1em]
&+&\displaystyle
\sum_{j} \< \Pi_m(\nabla_x(F_j)\cdot D^*)g_m,(F_j\cdot \tfrac{D+D^*}{2})\nabla_xg_m\>\\[1em]
&+&\displaystyle
\sum_{j} \< \Pi_m(F_j\cdot D^*)g_m,(\nabla_x(F_j)\cdot \tfrac{D+D^*}{2})\nabla_xg_m\>\\[1em]
&+&\displaystyle
\frac12\sum_{j} \|\Pi_m(\nabla_x(F_j)\cdot D^*)g_m\|^2\,.
\end{array}
$$
As a result, using \eqref{DPI}, we obtain
\begin{multline}\label{hypo2}
\begin{aligned}
&\dfrac12\dfrac{\dd}{\dd t}\E\|\nabla_xg_m\|^2
+ \E\|D\nabla_xg_m\|^2\\
&\leq \frac{\lambda^2}{2}\E\big{[}\|D^*g_m\|^2 + (2\|D^*g_m\|+\|D^*\Pi_{m,m-1}\nabla_xg_m\|)\,(\|D\nabla_xg_m\|+\|D^*\Pi_{m,m-1}\nabla_xg_m\|)\big{]}
\end{aligned}
\end{multline}

\subsubsection{Third estimate: $\E\|D g_m\|^2$} 
\noindent Recalling $\{A,D\} = -\nabla_x$ and $\{D,L\}=-ND$, by choosing $S=T=D$, we derive
$$\dfrac12\dfrac{\dd}{\dd t}\E\|Dg_m\|^2=-\E\<\nabla_xg_m,Dg_m\> - \E\|D^2g_m\|^2 - N\E\|Dg_m\|^2 + \frac{\lambda^2}{2}\E\,\mathcal{N}^{(m)}_{D,D}(g_m).$$
Furthermore, we have
$$
\begin{array}{rcl}\displaystyle
\mathcal{N}^{(m)}_{D,D}(g_m)&=&\displaystyle
\sum_{j} \|D\Pi_m(F_j\cdot D^*)g_m\|^2\\[1em]
&+&\displaystyle
\sum_{j} \< (F_j\cdot D^*)\Pi_m(F_j\cdot D^*)g_m,D^*Dg_m\>\,
\end{array}
$$
and, by \eqref{DPI}, $\<\nabla_xg_m,Dg_m\>=\<\Pi_{m,m-1}\nabla_xg_m,Dg_m\>$.
It follows then (using some inequalities like $\|D^*Dh\|\leq \|(D^*)^2h\|$ and $\|\Pi h\|\leq \|h\|$ with $\Pi=\Pi_m$ or $\Pi=\Pi_{m,m-1}$) that
\begin{equation}\label{hypo3}
\dfrac12\dfrac{\dd}{\dd t}\E\|D g_m\|^2+\E\|D^2 g_m\|^2\leq
\E\|\Pi_{m,m-1}\nabla_xg_m\|\|Dg_m\| +\frac{\lambda^2}{2}\E\big{[}\|(D^*)^2g_m\|^2+\|DD^*g_m\|^2\big{]}.
\end{equation}

\subsubsection{Fourth estimate: $\E\langle\nabla_x g_m,Dg_m\rangle$} 

\noindent We apply \refe{ST1} with $S=\nabla_x$ and $T=D$. It yields
$$
\begin{aligned}
\dfrac{\dd}{\dd t}\E\langle\nabla_x g_m,Dg_m\rangle &=-\E\<\nabla_x\Pi_mAg_m,Dg_m\> -\E\<D\Pi_mAg_m,\nabla_xg_m\> \\
&+ \E\<\nabla_xLg_m,Dg_m\> + \E\<DLg_m,\nabla_xg_m\> \\
&+\frac{\lambda^2}{2}\E[\mathcal{N}^{(m)}_{\nabla_x,D}(g_m)+\mathcal{N}^{(m)}_{D,\nabla_x}(g_m)].
\end{aligned}
$$
For the total equation, \textit{i.e.} when there is no projector $\Pi_m$, one would use
\begin{equation}\label{missingtotal}
-\<\nabla_x Ag,Dg\>-\<DAg,\nabla_x g\>=-\|\nabla_x g\|^2.
\end{equation}
A way to derive\footnote{Obviously one may also use concrete definitions of differential operators but the abstract way shown here has a clearer counterpart at the Galerkin level.} \eqref{missingtotal} is to write 
$$
-\<\nabla_x Ag,Dg\>-\<DAg,\nabla_x g\>=\<(D^*-D)Ag,\nabla_x g\>
$$
and to use the identity $A=(D+D^*)\cdot\nabla_x$. This gives
$$
-\<\nabla_x Ag,Dg\>-\<DAg,\nabla_x g\>=\<\{D^*,D\}\nabla_x g,\nabla_x g\>,
$$
and one concludes by use of the identity $\{D^*,D\}=-\mathrm{Id}$. For the terms with projectors, the same kind of computations gives
$$
-\<\nabla_x\Pi_mAg_m,Dg_m\> -\<D\Pi_mAg_m,\nabla_xg_m\>=\<\{D^*\Pi_m,D\Pi_m\}\nabla_x g,\nabla_x g\>.
$$
By \eqref{DPI} and the identity $\Pi_m\Pi_{m,m-1}=\Pi_{m,m-1}$, it follows that
\begin{multline*}
-\<\nabla_x\Pi_mAg_m,Dg_m\> -\<D\Pi_mAg_m,\nabla_xg_m\>\\
= -\|\Pi_{m,m-1}\nabla_xg_m\|^2\ +\ \<D\nabla_x g_m,D(\Pi_m-\Pi_{m,m-1})\nabla_xg_m\>
\end{multline*}
Besides, identity $L = -D^*D = N\textrm{Id}-DD^*$ provides
$$
\begin{aligned}
\E\<\nabla_xLg_m,Dg_m\> + \E\<DLg_m,\nabla_xg_m\> &= -\E\<D^*D\nabla_xg_m,Dg_m\> - \E\<DD^*D g_m,\nabla_xg_m\> \\
& \!\!\!\!\!\!\!\!\!\!\!\!\!\!\!\!\!\!\!\!\!\!\!\!\!\!\!\!\!\!\!\!\!\!\!\!\!= -\E\<D\nabla_xg_m,D^2g_m\> - \E\<D^*DD g_m,\nabla_xg_m\> - N\E\<Dg_m,\nabla_xg_m\>  \\
& \!\!\!\!\!\!\!\!\!\!\!\!\!\!\!\!\!\!\!\!\!\!\!\!\!\!\!\!\!\!\!\!\!\!\!\!\!= -2\E\<D\nabla_xg_m,D^2g_m\> - N\E\<Dg_m,\Pi_{m,m-1}\nabla_xg_m\> .
\end{aligned}
$$
Concerning the terms $\mathcal{N}^{(m)}_{\nabla_x,D}(g_m)$ and $\mathcal{N}^{(m)}_{D,\nabla_x}(g_m)$, we write them as the sum of
$$
-\sum_{j} \<(\Pi_m F_j\cdot D^*)^2g_m,D\cdot\nabla_xg_m\>+\<\nabla_x \Pi_m(F_j\cdot D^*)g_m,D\Pi_m(F_j\cdot D^*)g_m\>
$$
and 
$$
\sum_{j} \<(\Pi_m F_j\cdot D^*)^2g_m,D^*\cdot\nabla_xg_m\>+\<D \Pi_m(F_j\cdot D^*)g_m,\nabla_x\Pi_m(F_j\cdot D^*)g_m\>,
$$
to bound them proceeding as before by the sum of terms
$$
\|(D^*)^2g_m\|\|D\nabla_xg_m\|,\quad 2(\|D^*\Pi_{m,m-1}\nabla_xg_m\|+\|D^*g_m\|)\|DD^*g_m\|
$$
and 
$$
\|(D^*)^2g_m\|\|D^*\Pi_{m,m-1}\nabla_xg_m\|\,.
$$
As a result, we finally obtain
\begin{equation}\label{hypo4}
\begin{aligned}
\dfrac{\dd}{\dd t}\E\langle\nabla_x g_m,Dg_m\rangle&+\E\|\Pi_{m,m-1}\nabla_x g_m\|^2 \\
&\leq \E\|D\nabla_xg_m\|^2+2\E\|D\nabla_xg_m\|\|D^2g_m\|+N\E\|Dg_m\|\|\Pi_{m,m-1}\nabla_xg_m\| \\
&\quad+\frac{\lambda^2}{2}\E\big{[}
\|(D^*)^2g_m\|\ (\|D\nabla_xg_m\|+\|D^*\Pi_{m,m-1}\nabla_xg_m\|)\\
&\qquad\qquad+2\|DD^*g_m\|\ (\|D^*\Pi_{m,m-1}\nabla_xg_m\|+\|D^*g_m\|)
\big{]}.
\end{aligned}
\end{equation}

Let us summarize in the following proposition the differential inequalities derived above.
\begin{proposition} Suppose that hypothesis $(\ref{noise})$ holds and let $g_\mathrm{in}\in L^2(\Omega;L^2_{x,v})$. Let $m\in\N$ and let $g_m$ denote the Galerkin approximation given by Proposition~\ref{p:exGalerkin}. Then we have the following estimates, respectively on
\smallskip

\noindent $\bullet$ the $L^2_{\omega,x,v}$-norm of $g_m$
$$
\dfrac12\dfrac{\dd}{\dd t}\E\|g_m\|^2+\E\|D g_m\|^2\leq \lambda^2\E\|D^*g_m\|^2\,,
$$
$\bullet$ the $L^2_{\omega,x,v}$-norm of $\nabla_x g_m$
$$
\begin{aligned}
\dfrac12&\dfrac{\dd}{\dd t}\E\|\nabla_xg_m\|^2
+ \E\|D\nabla_xg_m\|^2\\
&\leq \frac{\lambda^2}{2}\E\big{[}\|D^*g_m\|^2 + (2\|D^*g_m\|+\|D^*\Pi_{m,m-1}\nabla_xg_m\|)\,(\|D\nabla_xg_m\|+\|D^*\Pi_{m,m-1}\nabla_xg_m\|)\big{]},
\end{aligned}
$$
$\bullet$ the $L^2_{\omega,x,v}$-norm of $D g_m$
$$
\dfrac12\dfrac{\dd}{\dd t}\E\|D g_m\|^2+\E\|D^2 g_m\|^2\leq
\E\|\Pi_{m,m-1}\nabla_xg_m\|\|Dg_m\| +\frac{\lambda^2}{2}\E\big{[}\|(D^*)^2g_m\|^2+\|DD^*g_m\|^2\big{]},
$$
$\bullet$ the $L^1_{\omega}$-norm of the cross product $\<\nabla_x g_m,D g_m\>$
$$
\begin{aligned}
\dfrac{\dd}{\dd t}\E\langle\nabla_x g_m,Dg_m\rangle&+\E\|\Pi_{m,m-1}\nabla_x g_m\|^2 \\
&\leq \E\|D\nabla_xg_m\|^2+2\E\|D\nabla_xg_m\|\|D^2g_m\|+N\E\|Dg_m\|\|\Pi_{m,m-1}\nabla_xg_m\| \\
&\quad+\frac{\lambda^2}{2}\E\big{[}
\|(D^*)^2g_m\|\ (\|D\nabla_xg_m\|+\|D^*\Pi_{m,m-1}\nabla_xg_m\|)\\
&\qquad\qquad+2\|DD^*g_m\|\ (\|D^*\Pi_{m,m-1}\nabla_xg_m\|+\|D^*g_m\|)
\big{]}.
\end{aligned}
$$
\end{proposition}
\subsubsection{Closed form of the estimates} 

\noindent In this section, we gather estimates \refe{hypo1}, \refe{hypo2}, \refe{hypo3} and \refe{hypo4} --- derived above --- in a closed form with respect to $g_m$, $\nabla_x g_m$, $Dg_m$, $D\nabla_xg_m$ and $D^2g_m$. Note in particular that we need to replace all occurrences of the operator $D^*$ using formula
\begin{equation}\label{DstarD0}
\|D^* f\|^2=\|D f\|^2+N\|f\|^2
\end{equation}
proved by \refe{Df} and \refe{Dstarf}. In what follows $C$ denotes a positive constant that depends only on the dimension $N$.

{\em First estimate.} The first bound \eqref{hypo1} can now be written as
\begin{equation}\label{hypo1closed}
\dfrac12\dfrac{\dd}{\dd t}\E \|g_m\|^2+\E\|Dg_m\|^2 \leq C\,\lambda^2\E[\|Dg_m\|^2+\|g_m\|^2].
\end{equation}

{\em Second estimate.} The second one \eqref{hypo2} becomes
\begin{equation}\label{hypo2closed}
\begin{aligned}
\dfrac12 &\dfrac{\dd}{\dd t}\E\|\nabla_xg_m\|^2 + \E\|D\nabla_xg_m\|^2  \\
&\leq C\,\lambda^2\E\big{[}\|g_m\|^2\,+\,\|Dg_m\|^2\,+\,\|\Pi_{m,m-1}\nabla_xg_m\|^2
\,+\,\|D\nabla_xg_m\|^2\big{]}.
\end{aligned}
\end{equation}

{\em Third estimate.} 
Concerning the third one \eqref{hypo3}, we obtain
\begin{equation}\label{hypo3closed}
\dfrac12\dfrac{\dd}{\dd t}\E\|Dg_m\|^2+\E\|D^2g_m\|^2 \leq
\E\|Dg_m\|\|\Pi_{m,m-1}\nabla_xg_m\| +C\,\lambda^2\E\big{[}\|g_m\|^2+\|Dg_m\|^2+\|D^2g_m\|^2\big{]}.
\end{equation}

{\em Fourth estimate.} Finally, likewise, the fourth bound \eqref{hypo4} writes
\begin{equation}\label{hypo4closed}
\begin{aligned}
&\dfrac{\dd}{\dd t}\E\langle \nabla_xg_m,Dg_m \rangle+\E\|\Pi_{m,m-1}\nabla_xg_m\|^2\\ 
&\leq \E\|D\nabla_xg_m\|^2+2\E\|D\nabla_xg_m\|\|D^2g_m\|+N\E\|Dg_m\|\|\Pi_{m,m-1}\nabla_xg_m\| \\
 &\quad+C\,\lambda^2\E\big{[}\|g_m\|^2\,+\,\|Dg_m\|^2\,+\,\|\Pi_{m,m-1}\nabla_xg_m\|^2
\,+\,\|D\nabla_xg_m\|^2\,+\,\|D^2\nabla_xg_m\|^2\big{]}.
\end{aligned}
\end{equation}

\subsection{Hypocoercive estimates} \label{sec:hypo}
\noindent In this section, we derive hypocoercive estimates \eqref{Hypolow} and \eqref{Hypostrong}. Without loss of generality we assume $t_0=0$ and $g_\mathrm{in}\in L^2(\Omega;L^2_{\nabla,D})$. Our strategy is to prove uniform bounds on the approximate solutions $(g_m)_m$ and take the limit $m\to\infty$. 

\subsubsection{Balance of the estimates}
\noindent To prove an exponential damping we shall combine \eqref{hypo1closed}, \eqref{hypo2closed}, \eqref{hypo3closed} and \eqref{hypo4closed} of Section \ref{sec:estimates} to identify a functional bounded by its own dissipation. The first step is to explain how to bound $\|g_m\|$. Mark that when $m\geq1$
\begin{equation}\label{e:Poincaré}
\begin{aligned}
\|g_m\|^2&\ =\ \sum_{\substack{|k|\leq m\\|l|\leq m}}\ |\<e_{k,l},g_m\>|^2\\
&\leq\ \sum_{\substack{|k|\leq m\\0<|l|\leq m}}\ |l|^2\,|\<e_{k,l},g_m\>|^2
\ +\ \sum_{\substack{0<|k|\leq m}}\ (2\pi|k|)^2|\<e_{k,0},g_m\>|^2
\ +\ \ |\<e_{0,0},g_m\>|^2\\
&\leq \|Dg_m\|^2\ +\ \|\Pi_{m,m-1}\nabla_xg_m\|^2\ +\ |\rho_\infty(g_m)|^2\,.
\end{aligned}
\end{equation}

Now we look for a suitable functional in the form 
$$
\mathcal{F}(g)=\|g\|^2+\alpha\|\nabla_x g\|^2 + \beta\|Dg\|^2 + 2\gamma\langle \nabla_x g,Dg\rangle.
$$
where $\alpha$, $\beta$, $\gamma$ are some positive coefficients. First we require $\gamma^2<\alpha\beta$ so as to ensure
\begin{equation}
\label{comparaison}
C_1\|g\|^2_{L^2_{\nabla,D}} \leq\mathcal{F}(g) \leq C_2\|g\|^2_{L^2_{\nabla,D}},
\end{equation}
for some positive constants $C_1,\,C_2$.

Now by adding \eqref{hypo1closed}, \eqref{hypo2closed}, \eqref{hypo3closed} and \eqref{hypo4closed}, we have
$$
\begin{aligned}
\dfrac12&\dfrac{\dd}{\dd t}\E\mathcal{F}(g_m)\
+\ \E\Big[\|Dg_m\|^2+\alpha\|D\nabla_xg_m\|^2+\beta\|D^2g_m\|^2+\gamma\|\Pi_{m,m-1}\nabla_xg_m\|^2\Big]\\[0.5em]
&\leq (1+\alpha+\beta+\gamma)\,C\lambda^2 \E|\rho_\infty(g_\mathrm{in})|^2\\[0.5em]
&\quad+\ (\beta+N\gamma)\E\|Dg_m\|\|\Pi_{m,m-1}\nabla_xg_m\|+\gamma\E\|D\nabla_xg_m\|^2+2\gamma\E\|D\nabla_xg_m\|\|D^2g_m\|\\[0.5em]
&\quad+\ (1+\alpha+\beta+\gamma)\,2C\lambda^2
\E\big{[}\|Dg_m\|^2\,+\,\|\Pi_{m,m-1}\nabla_xg_m\|^2
\,+\,\|D\nabla_xg_m\|^2\,+\,\|D^2\nabla_xg_m\|^2\big{]}
\end{aligned}
$$
from which follows 
\begin{equation}\label{hypofin1}
\begin{aligned}
\dfrac12&\dfrac{\dd}{\dd t}\E\mathcal{F}(g_m)\
+\ K\E\Big[\|Dg_m\|^2+\|D\nabla_xg_m\|^2+\|D^2g_m\|^2+\|\Pi_{m,m-1}\nabla_xg_m\|^2\Big]\\[0.5em]
&\leq K'\lambda^2 \E|\rho_\infty(g_\mathrm{in})|^2\ +\ K'\lambda^2
\E\big{[}\|Dg_m\|^2+\|\Pi_{m,m-1}\nabla_xg_m\|^2+\|D\nabla_xg_m\|^2\,+\,\|D^2\nabla_xg_m\|^2\big{]}
\end{aligned}
\end{equation}
for some positive $K$, $K'$ depending only on $N$, $\alpha$, $\beta$ and $\gamma$, provided that
$\gamma\leq \alpha/2$ and both $(\beta+N\gamma)/\sqrt{1\times\gamma}$ and $\gamma/\sqrt{\alpha\times\beta}$ are sufficiently small. The latter constraints may be satisfied jointly with $\gamma^2<\alpha\beta$ by setting\footnote{There is of course no uniqueness in this choice. For instance setting $\alpha=1$, $\beta=\gamma^\theta$, any $\frac12<\theta<2$ would work provided $\gamma$ is small enough.} $\alpha=1$, $\beta=\gamma$ and choosing $\gamma$ sufficiently small.

Having picked suitable parameters $\alpha$, $\beta$, $\gamma$, we now require $\lambda$ to be sufficiently small --- in a way that depends only on $N$ --- to derive
\begin{equation}\label{hypofin2}
\begin{aligned}
\dfrac12&\dfrac{\dd}{\dd t}\E\mathcal{F}(g_m)\
+\ K''\E\Big[\|g_m\|^2+\|Dg_m\|^2+\|D\nabla_xg_m\|^2+\|D^2g_m\|^2+\|\Pi_{m,m-1}\nabla_xg_m\|^2\Big]\\[0.5em]
&\leq K''' \E|\rho_\infty(g_\mathrm{in})|^2
\end{aligned}
\end{equation}
for some positive constants $K''$, $K'''$ depending only on $N$.

\subsubsection{Exponential damping} 
\noindent 
Integrating \eqref{hypofin2} from $0$ to $t$ and taking the limit $m\to\infty$ yields \eqref{Hypolow} (for $t_0=0$).

To prove \eqref{Hypostrong} we first stress that proceeding as in the proof of \eqref{e:Poincaré} gives
$$
\|\nabla_xg_m\|^2\ \leq\ \|\Pi_{m,m-1}\nabla_xg_m\|^2\ +\ \|D\nabla_xg_m\|^2
$$
and conclude then from \eqref{hypofin2} and \eqref{comparaison} that
$$
\frac12\dfrac{\dd}{\dd t}\E\mathcal{F}(g_m)\ +\ c\,\E\mathcal{F}(g_m)\ \leq\ C\,\E|\rho_\infty(g_\mathrm{in})|^2
$$
for some positive constants $c$ and $C$. This yields
$$
\forall t\geq0\,,\quad \E\mathcal{F}(g_m)(t)\ \leq\ \E\mathcal{F}(g_\mathrm{in})e^{-2\,ct}
\ +\ \frac{C}{c}\,\E|\rho_\infty(g_\mathrm{in})|^2
$$
through a multiplication by $e^{2ct}$ and an integration in time. Using again \eqref{comparaison} and taking the limit $m\to\infty$ achieves the proof of \eqref{Hypostrong} (for $t_0=0$).

\subsection{Hypoelliptic estimates}\label{sec:regul}

\noindent In this part, we conclude the proof of Theorem~\ref{thmhypo} by showing that the solution $g$ to Equation \eqref{eqg} with initial condition $g_{\textrm{in}}$ in $L^2(\Omega;L^2_{x,v})$  gains regularity instantaneously. Precisely, we prove the following result.

\begin{proposition}\label{propregul}
Let $t_0>0$. There exist positive constants $\lambda^*$ and $C$ such that for any $g_\mathrm{in}\in L^2(\Omega;L^2_{x,v})$ and $|\lambda|<\lambda^*$, the corresponding solution $g$ satisfies for any $t\in(0,t_0]$ 
\begin{equation}\label{propregulb}
\E\|g(t)\|^2 \leq C\E\|g_{\text{in}}\|^2,\quad \E\|Dg(t)\|^2 \leq \frac{C}{t}\E\|g_{\text{in}}\|^2,\quad \E\|\nabla_xg(t)\|^2 \leq \frac{C}{t^3}\E\|g_{\text{in}}\|^2.
\end{equation}
\end{proposition}
\noindent By a simple approximation argument one may reduce the proof of the proposition to the proof of estimates \eqref{propregulb} starting from $g_{\text{in}}\in L^2_{\nabla,D}$. For writing convenience, we assume $t_0=1$, modifications to obtain the proof of the general case being mostly notational. 

Though the proof of Proposition~\ref{propregul} has some similarities with the proof of exponential damping, constraints on functionals leading to global hypoelliptic estimates are a lot more stringent\footnote{This may be seen on the fact that in the strategy hereafter estimates should be compatible with chosen powers of~$t$.} and we have not been able to produce them entirely at the level of the Galerkin approximation. Instead we directly derive estimates on $g$ by examining the equations satisfied by $\Pi_m g$ and taking the limit $m\to\infty$ using the already established propagation of regularity. The key gain is that terms analogous to $\E\|D\nabla_xg_m\|^2$ in \eqref{hypo4closed} that arises from failure of commutativity of $\Pi_m$ and $D^*$ disappear when applied to $g$ in the limit $m\to\infty$ because $\{\Pi_m,D^*\}=-D^*(\Pi_m-\Pi_{m,m-1})$. 

We introduce the family of functionals parametrized by $t\in[0,1]$,
$$\mathcal{K}_t(g) :=\|g\|^2 + at^3\|\nabla_xg\|^2 + bt\|Dg\|^2 + 2ct^2\<\nabla_xg,Dg\>$$
where $a$, $b$ and $c$ are some positive constants to be chosen later on. By requiring $c^2<ab$, we ensure
\begin{equation}\label{compinfK}
\|g\|^2+C_1(t^3\|\nabla_x g\|^2+t\|Dg\|^2)\leq\mathcal{K}_t(g)\leq\|g\|^2+C_2(t^3\|\nabla_x g\|^2+t\|Dg\|^2)
\end{equation}
for some positive $C_1$, $C_2$. Proceeding as explained above we derive\footnote{The reader is referred to the treatment of a similar case in \cite[Appendix~A.21]{villani} for omitted details concerning algebraic manipulations.} for any $0\leq t\leq1$
$$
\mathcal{K}_t(g(t))\ +\ C\int_0^t (\E\|g(s)\|^2+s^3\E\|D\nabla_xg(s)\|^2 +s\E\|D^2g(s)\|^2 + s^2\|\nabla_x g(s)\|^2)\dd s\ \leq\ \mathcal{K}_0(g_{\text{in}})
$$
for some positive $C$, provided first that $a$, $b$ and $c$ are chosen such that both $(b+c)/\sqrt{1\times c}$ and 
$c/\sqrt{a\times b}$ are sufficiently small and then that $\lambda$ is sufficiently small. As above constraints on $a$, $b$, $c$ may be fulfilled by choosing $a=1$, $b=c$ and $c$ small enough. By appealing to \eqref{compinfK} we achieve the proofs.

\section{Invariant measure}\label{sec:invmeas}

\noindent In this section, we prove Theorem~\ref{thminv}. To do so, we fix $\bar{\rho}\in\R$ and assume $|\lambda|<\lambda_0$.

\subsection{Proof of existence}

Let $g_{\text{in}}\in L^2_{x,v}$ be a deterministic initial datum in $X_{\bar{\rho}}$. We consider the unique solution $g$ to problem (\hyperref[probm]{$\text{P}_{\bar{\rho}}$}) given by Theorem \ref{thmex}. First of all, using the regularizing bound \eqref{thmregul} of Theorem \ref{thmhypo}, we deduce that there exists a positive constant $C$ such that 
\begin{equation}\label{invregultsmall}
\E\|g(1)\|^2_{L^2_{\nabla,D}}\leq C\E\|g_{\text{in}}\|^2.
\end{equation}
We also recall the damping estimate \eqref{Hypostrong} of Theorem \ref{thmex}: for $t\geq 1$,
$$
\E\|g(t)\|_{L^2_{\nabla,D}}^2  \leq\ Ce^{-c(t-1)}\E\|g(1)\|_{L^2_{\nabla,D}}^2 + K \E|\rho_\infty(g)|^2.
$$
It implies, with \eqref{invregultsmall},
\begin{equation}\label{boundforex}
\sup\limits_{t\geq 1}\E\|g(t)\|_{L^2_{\nabla,D}}^2  \leq\ C\E\|g_{\text{in}}\|^2 + K {\bar{\rho}}^2.
\end{equation}
We introduce the family $(\mu_T)_{T>0}$ of probability measures on $X_{\bar{\rho}}$ defined by
$$
\mu_T:=\frac{1}{T}\int_{1}^{1+T}\!\!\!\!\mathscr{L}(g(t))\,\dd t,
$$
where $\mathscr{L}(g(t))$ denotes the law of $g(t)$, and show that the family $(\mu_T)_{T>0}$ is tight. Since the embedding $L^2_{\nabla,D}\subset L^2_{x,v}$ is compact, balls of radius $R>0$
$$K_R:=\{f\in X_{\bar{\rho}};\,\|f\|_{L^2_{\nabla,D}}\leq R\}$$
are compact in $X_{\bar{\rho}}$. Furthermore, thanks to Markov's inequality and \eqref{boundforex},
$$
\begin{aligned}
\mu_T(K_R^c)\ & =\ \frac{1}{T}\int_{1}^{1+T}\!\!\!\!\PP(\|g(t)\|_{L^2_{\nabla,D}}>R)\,\dd t \\
& \leq\ \frac{1}{TR^2}\int_{1}^{1+T}\!\!\!\!\E\|g(t)\|^2_{L^2_{\nabla,D}}\,\dd t \\
& \leq\ \frac{1}{R^2}(C\E\|g_{\text{in}}\|^2 + K {\bar{\rho}}^2). 
\end{aligned}
$$
This readily implies tightness of $(\mu_T)_{T>0}$. By Prokhorov's Theorem, see for instance \cite[Theorem~2.3]{daprato}, we obtain that $(\mu_T)_{T>0}$ admits  a subsequence (still denoted $(\mu_T)$) such that $\mu_T$ converges to some probability measure $\mu$ on $X_{\bar{\rho}}$ as $T\to\infty$. Furthermore, a classical argument shows that this limit measure $\mu$ is indeed an invariant measure for problem (\hyperref[probm]{$\text{P}_{\bar{\rho}}$}), see for instance \cite[Proposition 11.3]{daprato}.\medskip

\subsection{Proof of the mixing property} 

Let $g_{\text{in},1}$ and $g_{\text{in},2}\in X_{\bar{\rho}}$ and denote by $g_1$ and $g_2$ the solutions to (\hyperref[probm]{$\text{P}_{\bar{\rho}}$}) with respective initial conditions $g_{\text{in},1}$ and $g_{\text{in},2}$. For $t\geq 0$ we set $r(t):=g_1(t)-g_2(t)$ and remark that $r$ solves ({$\text{P}_0$}) on $X_0$. Combining again \eqref{Hypolow} and \eqref{Hypostrong} and recalling that \eqref{Dstarf} yields $\frac{N}{2}\|f\|^2\leq \|f\|^2_{L^2_{\nabla,D}}$, we deduce that there exists positive constants $c$ and $C$ such that, for $t\geq 1$,
\begin{equation}\label{boundforuniq}
\E\|r(t)\|^2  \leq\ Ce^{-c(t-1)}\E\|g_{\text{in},1}-g_{\text{in},2}\|^2.
\end{equation}
Let $\Psi\,:\ X_{\bar{\rho}}\to \R$ be $1$-Lipschitz continuous, let $g_\text{in}\in X_{\bar{\rho}}$ and $s>0$. We apply \eqref{boundforuniq} with $g_{\text{in},1}=g_\text{in}$ and $g_{\text{in},2}=g(s)$ to obtain for any $t\geq1$, $T>0$,
\begin{align}
\left|\E\Psi(g(t))-\frac{1}{T}\int_1^{T+1}\E\Psi(g(t+s)) \dd s\right|^2&\leq\frac{1}{T}\int_1^{T+1}\E\|g(t)-g(t+s)\|^2 \dd s\nonumber\\
&\leq Ce^{-c(t-1)} \frac{1}{T}\int_1^{T+1}\E\|g_\text{in}-g(s)\|^2 \dd s.\label{mixmix}
\end{align}
By \eqref{thmregul}, we have $\sup_{s\in[1,T]}\|g(s)\|\leq C\|g_\text{in}\|$ uniformly in $T$ (for some possibly different $C$) and we deduce from \eqref{mixmix} (for yet other values of constants) that for any $t$ and $T$
$$
\left|\E\Psi(g(t))-\frac{1}{T}\int_1^{T+1}\E\Psi(g(t+s)) \dd s\right|^2 \leq Ce^{-c(t-1)} \|g_\text{in}\|^2,
$$
Taking the limit $T\to+\infty$ gives the mixing estimate \eqref{mixingmu} (hence also the uniqueness part of Theorem~\ref{thminv}). \qed\bigskip

\subsection{An explicit case}

If $\bar{\rho}=0$, then $\mu_0$ is the Dirac mass on the solution $0$. There is also a nontrivial case in which we can explicitly compute the invariant measure $\mu_{\bar{\rho}}$ and in particular check that some smallness condition on $\lambda$ is indeed necessary.

\begin{proposition} Assume $\bar{\rho}\not=0$. Assume that $W_t$ is an $N$-dimensional Brownian motion, \textit{i.e.} $F_j$ are constant in $x$ with value $0$ for $j>N$ and the $j$-th vector of the canonical basis of $\R^N$ for $j=1,\dots,N$. Let $V^\mathrm{stat}(t)$, normally distributed with variance $1$, denote the stationary solution to the Langevin equation 
$$
\dd V(t)=-V(t)\dd t+\sqrt{2}\dd W_t.
$$
Then $\mu_{\bar{\rho}}$, the invariant measure for Equation~\eqref{eqg} is the law of the invariant solution given by $(t,x,v)\mapsto\bar{\rho} g^\mathrm{stat}(t,x,v)$, where
\begin{equation}\label{fstat}
g^\mathrm{stat}(t,x,v)=\cM^{-1/2}(v)\cM\left(v-\frac{\lambda}{\sqrt{2}}V^\mathrm{stat}(t)\right),
\end{equation}
and where $\cM$ is the Maxwellian function defined by \eqref{def:maxwellian}.
\label{explicitmu}\end{proposition}

\textbf{Proof of Proposition~\ref{explicitmu}.} See Appendix~\ref{app:probaprobaVFP}. \qed\bigskip

\begin{remark}\label{rk:lambda1} It is clear on Formula~\eqref{fstat} that the stochastic Vlasov force term in \eqref{FP} has a direct influence on the localization properties in $v$ of the solution. We compute 
$$
\|g^\mathrm{stat}\|^2=e^{\frac{\lambda^2}{2}|V^\mathrm{stat}(t)|^2}.
$$
In particular, we have
$$
\E\|g^\mathrm{stat}\|^2=\int_{\R^N}e^{\frac{\lambda^2}{2}|w|^2-\frac12|w|^2} dw.
$$
This is finite if, and only if, $|\lambda|<1$: we recover the necessity of this restriction on the size of the noise made in the statement of Theorem~\ref{thmex}. Note however that, here, no further restriction of the type $|\lambda|<\lambda_0$ as in the statement of Theorem~\ref{thminv} is necessary to obtain an invariant measure with mixing properties.
\end{remark}

\appendix

\section{Background and introductory material}\label{app:background}

In the present Appendix we gather some background material. Though it is certainly useless to the expert it may provide the reader unfamiliar with some of the main notions underlying the present paper a smoother entering gate.

\subsection{A compendium on the stochastic integral}\label{app:ItoInt}

Let $(\beta(t))_{t\in[0,T]}$ be a Brownian motion over $(\Omega,\mathcal{F},\PP)$. The first obstacle to the definition of the stochastic integral
$$
I(g)=\int_0^T g(t) \dd\beta(t)
$$
is the lack of regularity of $t\mapsto\beta(t)$, which has almost-surely a regularity $(1/2)^-$: for all $\alpha\in [0,1/2)$, almost-surely, $\beta$ is in $\mathcal{C}^\alpha([0,T])$ and not in $\mathcal{C}^{1/2}([0,T])$. In particular, when $g=\beta$, $I(g)$ can not be defined as a Young's Integral since this would require precisely $\beta$ to be in $\mathcal{C}^\alpha$ with $\alpha>1/2$. Therefore, in that context, one has to expand the theory of Young's or Riemann -- Stieltjes' Integral. This is one of the purpose of rough paths' theory, \textit{cf.} \cite{FrizHairer2014}, but as we briefly sketch below, the original definition of $I(g)$ does not need rough paths' theory. It uses the probabilistic properties of the Brownian motion. 

Let $(\mathcal{F}_t)_{t\in[0,T]}$ be a given filtration: this is an increasing set of sub-$\sigma$-algebra of $\mathcal{F}$. A process $(X(t))_{t\in[0,T]}$ is said to be adapted to $(\mathcal{F}_t)_{t\in[0,T]}$ (or adapted for short, if there is no ambiguity) if, for each $t\in[0,T]$, $X(t)$ is $\mathcal{F}_t$-measurable. We assume that $(\beta(t))_{t\in[0,T]}$ is adapted to $(\mathcal{F}_t)_{t\in[0,T]}$. If $g$ is an $L^2$-elementary predictable process, which means 
\begin{equation}\label{elementaryP}
g(\omega,t)=\sum_{k=0}^{n-1} g_k(\omega)\mathbf{1}_{(t_k,t_{k+1}]},
\end{equation}
where $(t_k)_{0\leq k\leq n}$ is a partition of $[0,T]$ and each random variable $g_k$ is $\mathcal{F}_{t_k}$ measurable and in $L^2(\Omega)$, then $I(g)$ is defined as the Riemann sum
\begin{equation*}
I(g)=\sum_{k=0}^{n-1}g_k(\beta(t_{k+1})-\beta(t_k)).
\end{equation*}
The probabilistic properties of $\beta$ imply that $I(g)$ is well defined in $L^2(\Omega)$ and that $\E I(g)=0$ and
\begin{equation}\label{ItoIsometry}
\E |I(g)|^2=\E\int_0^T |g(t)|^2 \dd t.
\end{equation}
See \cite[Theorem~2.3]{ChungWilliams90}. The identity \eqref{ItoIsometry} means that the map
\begin{equation}\label{ItoMap}
I\colon\mathcal{E}_T\subset L^2(\Omega\times(0,T),\PP\times\mathcal{L})\to L^2(\Omega,\PP)
\end{equation}
is an isometry. In \eqref{ItoMap}, we have denoted by $\mathcal{E}_T$ the set of $L^2$-elementary predictable functions in the form \eqref{elementaryP} and by $\mathcal{L}$ the Lebesgue measure on $[0,T]$. 

The It\={o} stochastic integral is the extension of $I$ to the closure of $\mathcal{E}_T$ in $L^2(\Omega\times[0,T],\PP\times\mathcal{L})$. The last task in the definition of the stochastic integral is the identification of the closure of $\mathcal{E}_T$ (and also the identification of subsets of this closure). For this purpose, we introduce $\mathcal{P}_T$, the predictable sub-$\sigma$-algebra of $\mathcal{F}\times\mathcal{B}([0,T])$ generated by the sets $F_0\times\{0\}$, $F_s\times(s,t]$, where $F_0$ is $\mathcal{F}_0$-measurable, $0\leq s<t\leq T$ and $F_s$ is $\mathcal{F}_s$-measurable. We have denoted by $\mathcal{B}([0,T])$ the Borel $\sigma$-algebra on $[0,T]$. It is clear that each element in $\mathcal{E}_T$ is $\mathcal{P}_T$ measurable. The first result, \cite[Lemma~2.4]{ChungWilliams90}, is that the closure of $\mathcal{E}_T$ in $L^2(\Omega\times[0,T])$ is $L^2_{\mathcal{P}}(\Omega\times[0,T])$, the set of functions in $L^2(\Omega\times[0,T])$ which are equal $\PP\times\mathcal{L}$-a.e. to a $\mathcal{P}_T$-measurable function. In the core of our paper we also use implicitly other characterizations of $L^2_{\mathcal{P}}(\Omega\times[0,T])$, we refer the reader to \cite[Chapter~IV-5]{RevuzYor99} and \cite[Chapter~3]{ChungWilliams90} for more on the topic.\medskip

The Stratonovich stochastic integral
$$
\int_0^T g(t)\circ\dd\beta(t)
$$
corresponds instead to the extension of the application associating the value
$$
\sum_{k=0}^{n-1}\frac{g_k+g_{k+1}}{2}(\beta(t_{k+1})-\beta(t_k))
$$
to an $L^2$-elementary predictable process $g$ of the form \eqref{elementaryP}. In particular for such a $g$ one may check readily that
\begin{equation}\label{eq:conversion}
\int_0^T g(t)\circ\dd\beta(t)\,=\,\int_0^T g(t)\dd\beta(t)\,+\,\tfrac12\,[g,\beta]_T
\end{equation}
where $[\cdot,\cdot]_T$ denotes co-variation, which for $g$ as above is written as
$$
[g,\beta]_T\,=\,\sum_{k=0}^{n-1}(g_{k+1}-g_k)(\beta(t_{k+1})-\beta(t_k))\,.
$$
Conversion formula~\eqref{eq:conversion} extends to $L^2_{\mathcal{P}}(\Omega\times[0,T])$. As a consequence
$$
\dd X_t\,=\,f(t,X_t)\circ\dd\beta_t\,+\,g(t,X_t)\dd t
$$
is equivalently written as
$$
\dd X_t\,=\,f(t,X_t)\dd\beta_t\,+\,(g(t,X_t)+\tfrac12\d_Xf(t,X_t)\,f(t,X_t))\dd t\,.
$$
See for instance \cite[Chapter~3]{Oksendal} for further discussions on the Stratonovich integral, including comments on how it naturally arises from physical modeling considerations.

\subsection{Existence of invariant measures by the compactness method}\label{app:invariant}

We briefly recall here the basic principle underpinning the use of a compactness argument to prove the existence of an invariant measure, by examining the case of a deterministic time evolution. Incidentally we note that, though we restrict to the case where the underlying semi-group is $\R_+$, a similar strategy proves the existence of an invariant measure for the action of any locally compact group, so-called Haar measure in this context.

Let $X$ be a (locally compact Hausdorff) topological space and $\phi:\R_+\to \mathcal{C}(X)$ be a continuous semi-flow. Pick any Borel measure on $X$ and define, for $T>0$,
$$
\mu_T\,=\,\frac1T\,\int_0^T (\phi_t)_*(\mu_0) \dd t\,.
$$
Then for any $t_0$ observe that when $T\geq t_0$
$$
(\phi_{t_0})_*(\mu_T)-\mu_T\,=\,\frac1T\,\int_T^{T+t_0} (\phi_t)_*(\mu_0) \dd t
-\frac1T\,\int_0^{t_0} (\phi_t)_*(\mu_0) \dd t
$$
so that
$$
\|(\phi_{t_0})_*(\mu_T)-\mu_T\|\,\leq\,\frac{2t_0}{T}\|\mu_0\|\stackrel{T\to\infty}{\longrightarrow}0\,.
$$
Therefore any accumulation point of $(\mu_T)_{T>0}$ (in any reasonable topology) is an invariant measure and the compactness of those time averages is sufficient to prove the existence of such a measure.

The main point to establish in order to apply those arguments --- compactness of time averages --- is usually obtained, as we do in the present contribution, through Prokhorov's Theorem. See for instance \cite[Theorem~2.3]{daprato} for this theorem. For a thorough discussion and numerous illustrations of the compactness argument for infinite-dimensional stochastic evolutions we refer the reader to \cite[Chapter~11]{daprato}.

\subsection{Ellipticity and coercivity by global estimates}\label{app:ellipticity}

Our implementation of the compactness method sketched above relies on  hypocoercivity and hypoellipticity of the stochastic evolution encoded by equation~\eqref{FP}. Among many possible approaches we establish such properties by energy estimates.

As it may serve as a guide through technical details of our analysis, for the convenience of the reader we recall here classical coercive and elliptic global estimates where "hypo" global arguments originate. For simplicity we only discuss a deterministic case without forcing, obeying 
\begin{equation}
\d_tf\,=\,\Delta_xf
\end{equation}
for some $f:\R_+\times\T^N\to\R$, $(t,x)\mapsto f(t,x)$ starting from $f_0$ at time $0$. Concerning coercivity note that for any $t\geq0$, $\int_{\T^N}f(t,\cdot)=\int_{\T^N}f_0$ and
$$
\begin{array}{rl}
\displaystyle\frac12\frac{\dd}{\dd t}
\left\|f-\int_{\T^N}f(\cdot,x)\dd x\right\|_{L^2(\T^N)}^2(t)
&\displaystyle
=\,-\|\nabla_x f(t,\cdot)\|_{L^2(\T^N)}^2\\
&\displaystyle
\leq\,-(2\pi)^{2d}\quad\left\|f(t,\cdot)-\int_{\T^N}f(t,x)\dd x\right\|_{L^2(\T^N)}^2
\end{array}
$$
where we have used a Poincar\'e inequality. This proves for any $t\geq0$
$$
\left\|f(t,\cdot)-\int_{\T^N}f(t,x)\dd x\right\|_{L^2(\T^N)}
\,\leq\, e^{-(2\pi)^{2d}\,t}\left\|f_0-\int_{\T^N}f_0(x)\dd x\right\|_{L^2(\T^N)}\,.
$$
Concerning ellipticity the model basic estimate is as follows. For any $t\geq0$
\begin{multline*}
\frac12\frac{\dd}{\dd t}
\left(t\mapsto\left\|f(t,\cdot)-\int_{\T^N}f(t,x)\dd x\right\|_{L^2(\T^N)}^2+2t\,\left\|\nabla_x f(t,\cdot)\right\|_{L^2(\T^N)}^2\right)(t)\\
\,=\,-2t\,\left\|\nabla_x^2 f(t,\cdot)\right\|_{L^2(\T^N)}^2\,\leq\,0\,.
\end{multline*}
In particular for any $t\geq0$
$$
\left\|\nabla_x f(t,\cdot)\right\|_{L^2(\T^N)}
\,\leq\,\frac{1}{\sqrt{2t}}\left\|f_0-\int_{\T^N}f_0(x)\dd x\right\|_{L^2(\T^N)}\,.
$$
One may also combine both estimates to derive for some constant $C$ and any $t\geq0$
$$
\left\|\nabla_x f(t,\cdot)\right\|_{L^2(\T^N)}
\,\leq\,C\,\max\left(\left\{1,\frac{1}{\sqrt{t}}\right\}\right)\,e^{-(2\pi)^{2d}\,t}\,\left\|f_0-\int_{\T^N}f_0(x)\dd x\right\|_{L^2(\T^N)}\,.
$$

In foregoing computations we have argued essentially formally but an approximation argument, either of projection/Galerkin type as in the main core of our paper or based on cut-off/mollifiers as in Appendix~\ref{app:thmexFP}, may provide needed justifications. The goal of "hypo" theories is to provide replacements for the above when diffusion is only partial. This includes kinetic models, where diffusive mechanisms do not act directly on all variables and that fall directly in H\"ormander's and Villani's frameworks, but also compressible fluid models where dissipation do no act directly on all components of the solution, a case originally treated by the Kawashima theory.

\section{The Vlasov-Fokker-Planck operator}\label{app:probaVFP}

The probabilistic interpretation of the Vlasov-Fokker-Planck operator has already been discussed in the introduction of the paper (see \eqref{XVsto}). In this section, we provide more details about it and deduce some estimates on the Green kernel of the solution operator associated with the Fokker-Planck equation. These estimates are used in Appendix~\ref{app:thmexFP} to solve Equation~\eqref{FP}. \medskip

\subsection{Green kernel and probabilistic interpretation}\label{app:probaTFP}

Let us denote by $X=(x,v)$, $Y=(y,w)$ generic variables in $\T^N\times\R^N$.
Let $K_t^\#(X;Y)$ denote the kernel of the solution operator associated with the kinetic Fokker-Planck equation
\begin{equation}\label{kineticFP}
\partial_t f=\cQ(f)-v\cdot\nabla_x f=:\L_\mathrm{FP} f
\end{equation}
on $\T^N\times\R^N$. The function $X\mapsto K_t^\#(X,Y)$ is the density with respect to the Lebesgue measure on $\T^N\times\R^N$ of the law $\hat\mu_t^{(Y)}$ of the solution $X_t^{(Y)}$ to the SDE \eqref{XVsto} with $F(t,x)\equiv 0$, satisfying the Cauchy condition 
\begin{equation}\label{CYsto}
X_{t=0}^{(Y)}=Y.
\end{equation}
Since $K_t^\#(\cdot;Y)$ is a probability density, the map 
$$
K_t^\#\colon f\mapsto K_t^\# f,\qquad K_t^\# f(X)=\iint _{\T^N\times\R^N} K^\#_t(X;Y)\,f(Y)\,\dd Y
$$ 
is well defined as a continuous operator from $L^1$ to $L^1$ and
\begin{equation}\label{KtL1}
\|K_t^\#\|_{L^1\to L^1}=\sup_{Y}\iint_X K_t^\#(X;Y)\,\dd Y =1.
\end{equation}
Note also that
\begin{equation}\label{dualityKt}
\iint_{\T^N\times\R^N} K_t^\# f(X)\,\varphi(X)\,\dd X=\hat{\E}\iint _{\T^N\times\R^N} f(Y)\,\varphi(X_t^{(Y)})\,\dd Y,
\end{equation}
for all $f\in L^1(\T^N\times\R^N)$, $\varphi\in \mathcal{C}_b(\T^N\times\R^N)$. The explicit expression of $X_t^{(Y)}$ is
\begin{align}
x_t^{(Y)}\,=\ &y+(1-e^{-t})w+\int_0^t (1-e^{-(t-s)})\,\dd\hat B_s\,,\label{XtVtX}\\
v_t^{(Y)}\,=\ &e^{-t}w+\int_0^t e^{-(t-s)}\,\dd\hat B_s\,.\label{XtVtV}
\end{align}
In particular the change of variable from $Y$ to $\tilde{Y}=X_t^{(Y)}$ has Jacobian $e^{-Nt}$. Combined with \eqref{dualityKt} this yields the estimate 
\begin{equation}\label{KtLinfty}
\sup_{X}\iint_Y K_t^\#(X;Y)\,\dd Y=\|K_t^\#\|_{L^\infty\to L^\infty}\leq e^{Nt}.
\end{equation}
Indeed, we deduce from \eqref{dualityKt} that
\begin{align*}
\left|\iint_{\T^N\times\R^N} K_t^\# f(X)\,\varphi(X)\,\dd X\right|&\ \leq\ \|f\|_{L^\infty(\T^N\times\R^N)}\ \hat{\E}\iint _{\T^N\times\R^N} |\varphi(X_t^{(Y)})|\,\dd Y\\
&\ =\ e^{Nt}\ \|f\|_{L^\infty(\T^N\times\R^N)}\ \|\varphi\|_{L^1(\T^N\times\R^N)}
\end{align*}
for $f\in L^1\cap L^\infty(\T^N\times\R^N)$, $\varphi\in L^1\cap \mathcal{C}_b(\T^N\times\R^N)$. Thanks to \eqref{KtL1} and \eqref{KtLinfty}, from the Cauchy-Schwarz inequality we conclude
\begin{equation}\label{KtL2}
\|K_t^\#\|_{L^2\to L^2}\leq e^{Nt/2}
\end{equation}
since
$$
\iint_{\T^N\times\R^N} |K_t^\#  f(X)|^2\,\dd X\leq
e^{Nt}\iint_{\T^N\times\R^N} K_t^\# (|f|^2)(X)\,\dd X
\leq e^{Nt} \iint_{\T^N\times\R^N}|f|^2(X)\,\dd X
$$
for all $f\in L^2$. Alternatively, on this last step one may invoke directly an interpolation argument on Lebesgue spaces.
\medskip

For completeness, we also give an (almost) explicit expression of $K_t^\#$. Temporarily omitting the periodic identification, we observe that when $Y=0$, the process $(X_t^{(0)})$ is Gaussian with covariance matrix
\begin{equation}\label{def:Qt}
Q_t:=\begin{pmatrix}
\alpha_t & \delta_t \\
\delta_t & \gamma_t
\end{pmatrix}
\otimes\mathrm{I}_N,\quad 
\begin{pmatrix}
\alpha_t & \delta_t \\
\delta_t & \gamma_t
\end{pmatrix}
:=\begin{pmatrix}
\displaystyle \int_0^t|1-e^{-s}|^2 \dd s & \displaystyle\int_0^t e^{-s}(1-e^{-s})\dd s \\
\displaystyle \int_0^t e^{-s}(1-e^{-s})\dd s & \displaystyle\int_0^t e^{-2s} \dd s
\end{pmatrix}
\end{equation}
where matrix notation corresponds to the identification $\R^N\times\R^N=(\R\times\R)^N$.
Denoting by $p_t(X)$ the probability density of this Gaussian process, that is,
\begin{equation}\label{def:pt}
p_t(X)=\frac{1}{(2\pi)^N\det(Q_t)^{1/2}}\exp\left(-\frac12\langle Q_t^{-1}X,X\rangle\right),
\end{equation}
we have by \eqref{dualityKt} and \eqref{XtVtX}-\eqref{XtVtV},
\begin{multline*}
\iint_{\T^N\times\R^N} K_t^\# f(X)\,\varphi(X)\,\dd X\\
=\iint_{\R^N\times\R^N}\iint _{\T^N\times\R^N} f(Y)\,\varphi(y+(1-e^{-t})w+x,e^{-t}w+v))\,p_t(X) \,\dd Y\,\dd X\\
=\iint_{\R^N\times\R^N}\iint _{\T^N\times\R^N} f(Y)\,\varphi(Y)\,p_t(x-[y+(1-e^{-t})w],v-e^{-t}w) \,\dd Y\,\dd X\\
=\iint_{\T^N\times\R^N}\iint _{\T^N\times\R^N} f(Y)\,\varphi(Y)\,\left(\sum_{\ell\in\Z^N}p_t(x-[y+\ell+(1-e^{-t})w],v-e^{-t}w)\right)\,\dd Y\,\dd X
\end{multline*}
and thus $K_t^\#$ is the periodic version
\begin{equation}\label{perKt}
K_t^\#(X;Y)=\sum_{\ell\in\Z^N}K_t(X;y+\ell,w)
\end{equation}
of the Green kernel $K_t$ of the transport-Fokker-Planck equation \eqref{kineticFP} set on $\R^N\times\R^N$
\begin{equation}\label{explicitKt}
K_t(X;Y)=p_t(x-[y+(1-e^{-t})w],v-e^{-t}w),
\end{equation}
with $p_t$ defined in \eqref{def:pt}. 

\subsection{Derivation of explicit invariant measures.}\label{app:probaprobaVFP}

We now use the probabilistic interpretation recalled above to compute explicitly some invariant measures and obtain Proposition~\ref{explicitmu}. Here our explicit investigation of invariant measures does not involve limits of time averages but instead take the limit $t_0\to-\infty$ of laws of processes started at a time $t_0<0$ with a fixed initial datum. Note that in contrast with the time-average method that lends itself to a compactness argument this requires the convergence of the \emph{full} family indexed by $t_0$.

For expository purposes we begin by showing how the argument provides a steady solution of the deterministic equation, corresponding to $\lambda=0$. Choose an initial probability $\mu_0$ of density $f_0$. Let us temporarily fix $t_0<0$. It is convenient to introduce a Brownian motion $\hat B_t$ defined on $\R$ and set $(\hat B_{t_0,t})_{t\geq t_0}:=(\hat B_t-\hat B_{t_0})_{t\geq t_0}$. Then we extend \eqref{XtVtX}-\eqref{XtVtV} by drawing $(x_0,v_0)$ with law $\mu_0$ and setting
\begin{align}
x_{t_0,t}&=x_0+(1-e^{-(t-t_0)})v_0+\hat x_{t_0,t},\quad \hat x_{t_0,t}=\sqrt{2}\int_{t_0}^t (1-e^{-(t-s)}) \dd\hat B_{t_0,s}\,,\label{Xt0t}\\
v_{t_0,t}&=e^{-(t-t_0)}v_0+\hat v_{t_0,t},\qquad\qquad\qquad\hat v_{t_0,t}=\sqrt{2}\int_{t_0}^t e^{-(t-s)}\dd\hat B_{t_0,s}\,.\label{Vt0t}
\end{align}
The process $\hat X_{t_0,t}=(\hat x_{t_0,t},\hat v_{t_0,t})$ is a Gaussian process with covariance matrix $Q_{t-t_0}$ given by \eqref{def:Qt}. Let $f_{t_0,t}$ denote the density of the law of $X_{t_0,t}$. Then 
\begin{equation}\label{LargeTimet0}
f_{t_0,t}(x,v)\stackrel{t_0\to-\infty}{\longrightarrow} \left[\iint_{\T^N\times\R^N}f_0(Y)\dd Y\right] \cM(v)\,.
\end{equation}
This recovers the trivial fact that, for any $\bar{\rho}\in\R$, $(t,x,v)\mapsto \bar{\rho}\cM(v)$ solves \eqref{FP} when $\lambda=0$. Let us give more details on the sense of \eqref{LargeTimet0} and its proof thanks to the probabilistic interpretation of $K_t^\#$. We use \eqref{dualityKt} and write
$$
\E\varphi(X_{t_0,t}^{(Y)})=\E\left[\E\left[\varphi(x_{t_0,t},v_{t_0,t})|\sigma(\hat v_{t_0,t})\right]\right].
$$
By \eqref{def:Qt}, conditionally to $\hat v_{t_0,t}=v$, $\hat x_{t_0,t}$ is a Gaussian random variable with covariance $(\alpha_\tau-2\delta_\tau\gamma_\tau^{-1})\mathrm{I_N}$ and mean $\delta_\tau\gamma_\tau^{-1}v$, where $\tau=t-t_0$. Since $\alpha_\tau-2\delta_\tau\gamma_\tau^{-1}\sim\tau$ when $\tau\to+\infty$, we have 
$$
\E\left[\varphi(x_{t_0,t},v_{t_0,t})|\sigma(\hat v_{t_0,t})\right]\sim\int_{\T^N}\varphi(x,\hat v_{t_0,t})dx,
$$
a.s., when $t_0\to+\infty$. Besides, $\hat v_{t_0,t}$ converges in law to the centred Normal law of covariance $\mathrm{I}_N$, therefore
$$
\E\varphi(X_{t_0,t}^{(Y)})\to \iint_{\T^N\times\R^N}\varphi(x,v)\cM(v) dx dv,
$$
when $t_0\to+\infty$. By \eqref{dualityKt}, we deduce \eqref{LargeTimet0} in the weak sense, when tested against $\varphi$ a con\-ti\-nu\-ous and bounded function of $(x,t)$.\smallskip

We now relax the constraint $\lambda=0$ but restrict to the context of Proposition~\ref{explicitmu} where $W_t$ is an $N$-dimensional Brownian motion. One may proceed as above and introduce, with obvious notation,
\begin{align}
x_{t_0,t}&=x_0+(1-e^{-(t-t_0)})v_0+\hat x_{t_0,t}+z_{t_0,t}\,,\quad z_{t_0,t}=\lambda\int_{t_0}^t (1-e^{-(t-s)}) \dd W_{t_0,s}\,,\label{Xt0t2}\\
v_{t_0,t}&=e^{-(t-t_0)}v_0+\hat v_{t_0,t}+u_{t_0,t},\qquad\qquad\qquad\qquad u_{t_0,t}=\lambda\int_{t_0}^t e^{-(t-s)}\dd W_{t_0,s}.\label{Vt0t2}
\end{align}
The density $f_{t_0,t}$ of the law of $(x_{t_0,t},v_{t_0,t})$ (with respect to $\hat\omega$) solves \eqref{FP} and is given by
$$
\iint_{\T^N\times\R^N}f_{t_0,t}(X)\varphi(X) \dd X=\hat{\E}\varphi(x_{t_0,t},v_{t_0,t}).
$$
By using \eqref{LargeTimet0} and the convergence in law
$$
\lim_{t_0\to-\infty} u_{t_0,t}= \frac{\lambda}{\sqrt{2}}V^\mathrm{stat}(t),
$$
where $V^\mathrm{stat}(t)$, normally distributed with variance $1$, denotes the stationary solution to the Langevin equation 
$$
\dd V(t)=-V(t)\dd t+\sqrt{2}\dd W_t,
$$
we obtain
\begin{multline*}
\lim_{t_0\to-\infty} \iint_{\T^N\times\R^N}f_{t_0,t}(X)\varphi(X)\dd X\\
=\left[\iint_{\T^N\times\R^N}f_0(Y)\dd Y\right] \iint_{\T^N\times\R^N}\varphi(X)\cM\left(v-\frac{\lambda}{\sqrt{2}}V^\mathrm{stat}(t)\right) \dd X
\end{multline*}
for all $\varphi\in C_b(\T^N\times\R^N)$. This provides the invariant measure in \eqref{fstat}.

\begin{remark} In $L^p(\T^N\times\R^N)$, $p\in[1,+\infty]$, the norm of
$$
f^\mathrm{stat}(t,x,v)=\cM\left(v-\frac{\lambda}{\sqrt{2}}V^\mathrm{stat}(t)\right)
$$
is $1$, a.s. Using \eqref{KtL1} we can prove, then, that, for all $f\in L^1(\T^N\times\R^N)$, the measure with density $K_{t}^\#f$ with respect to $dX$ is converging weakly to the law of $f^\mathrm{stat}$. No smallness condition on $\lambda$ is necessary here. This should be contrasted with Remark~\ref{rk:lambda1}. 
\end{remark}

\section{Proof of Theorem \ref{thmexFP}}\label{app:thmexFP}

\subsection{Approximation}\label{app:thmexFPapprox}

Let $(\psi_\delta)$ be an approximation of identity for the convolution on $\R^N$ in the form $\psi_\delta(v)=\delta^{-N}\psi(\delta^{-1}v)$, where $\psi$ is the smooth density of a probability measure on $\R^N$, compactly supported in $B(0,1)$. We also assume that $\psi$ is radially symmetric. Our aim is first to solve the regularized equation
\begin{equation}\label{FPdelta}
\dd f^\delta\ +\ v\cdot\nabla_x f^\delta\ \dd t\ +\ \lambda J^\delta(\nabla_v f^\delta)\odot\dd W_t^\delta\ =\ \cQ(f^\delta)\ \dd t,
\end{equation}
where $J^\delta$ is the convolution in $v$ with $\psi_\delta$ and $\dd W_t^\delta$ has the same expression as $\dd W_t$ but with $F_j^\delta$ replacing $F_j$, where $F_j^\delta$ is a smooth approximation of $F_j$ (smoothness of $F_j^\delta$ being required in Proposition~\ref{prop:regfdelta}), that we shall choose explicitly when taking the limit $\delta\to0$. We build mild solutions to \eqref{FPdelta}. Recall that $(K_t^\#)_{t\geq 0}$ denotes the Green Kernel of the transport-Fokker-Planck operator defined in \eqref{kineticFP} (see Section~\ref{app:probaTFP}) and that we also use the notation $K_t^\#$ to denote the operator $f\mapsto K_t^\# f$, where
\begin{equation}\label{actionKt}
K_t^\# f(X)=\iint_{\T^N\times\R^N}K_t^\# (X;Y)f(Y)\dd Y.
\end{equation}
At last, for concision's sake we also set $H=L^2(\T^N\times\R^N)$.

\begin{definition}[Mild solution to \eqref{FPdelta}] A function $f^\delta\in \mathcal{C}([0,T];L^2(\Omega;H))$ is said to be a mild solution to Equation~\eqref{FPdelta} with initial datum $f^\delta_\mathrm{in}\in H$ if the $H$-valued process $(f^\delta(t))_{t\geq0}$ is adapted and
\begin{align}
f^\delta(t)=K_t^\# f^\delta_\mathrm{in}-\lambda\sum_{j\geq 0}&\int_0^t K_{t-s}^\# (F_j^\delta\cdot \nabla_v J^\delta(f^\delta))(s) \dd\beta_j(s)\nonumber\\
&+\frac{\lambda^2}{2}\sum_{j\geq 0}\int_0^t K_{t-s}^\#(F_j^\delta\cdot \nabla_v J^\delta(F_j^\delta\cdot \nabla_v J^\delta (f^\delta)))(s)\dd s,\label{mildFPdelta}
\end{align}
for all $t\in[0,T]$.
\label{MildDelta}\end{definition}
Note that we consider the It\={o} form of \eqref{FPdelta} in \eqref{mildFPdelta}.\smallskip

Now we prove  the existence of a solution $f^\delta$ to \eqref{FPdelta} (Proposition~\ref{prop:Efdelta}) and show that it is also a weak solution to \eqref{FPdelta} (Proposition~\ref{prop:weakfdelta}). The natural energy estimate for $f^\delta$ shall provide uniform bounds that are sufficient to take the limit $\delta\to0$ in the weak formulation of the problem (Proposition~\ref{prop:Udelta}). As an intermediate step, to justify computations leading to this energy estimate we prove some regularity for $f^\delta$ in Proposition~\ref{prop:regfdelta}.

\begin{proposition}[Resolution of \eqref{FPdelta}] Let $f^\delta_\mathrm{in}\in H$. There exists a unique mild solution $f^\delta\in \mathcal{C}([0,T];L^2(\Omega;H))$ to \eqref{FPdelta} with initial datum $f^\delta_\mathrm{in}$. 
\label{prop:Efdelta}\end{proposition}

\textbf{Proof of Proposition~\ref{prop:Efdelta}.} Let $D^\delta_j$ denote the operator of convolution in $v$ with $F_j^\delta\cdot\nabla\psi_\delta$ 
\begin{equation}\label{defHdeltaj}
D^\delta_j f(x,v)=\int_{\R^N}F_j^\delta(x)\cdot\nabla\psi_\delta(w)f(x,v-w)\dd w.
\end{equation}
Then $F_j^\delta\cdot\nabla_vJ^\delta=D^\delta_j$ and \eqref{mildFPdelta} reads
\begin{equation}\label{mildFPdeltaH}
f^\delta(t)=K_t^\# f^\delta_\mathrm{in}-\lambda\sum_{j\geq 0}\int_0^t K_{t-s}^\# D^\delta_j f^\delta(s) \dd \beta_j(s)
+\frac{\lambda^2}{2}\sum_{j\geq 0}\int_0^t K_{t-s}^\# \left[D^\delta_j\right]^2  f^\delta(s)\dd s=:\mathcal{I}^\delta(f^\delta)(t).
\end{equation}
The operator $D^\delta_j$ is of order $0$. This is sufficient to solve the fixed-point equation \eqref{mildFPdeltaH} in the space $E_T$ of functions in $\mathcal{C}([0,T];L^2(\Omega;H))$ that are adapted. To prove this claim we first observe that
\begin{equation}\label{boundHdeltaj0}
|D^\delta_j f(x,v)|^2\leq \|\nabla\psi\|_{L^1(\R^N)}\frac{1}{\delta}|F_j^\delta(x)|^2\int_{\R^N}|\nabla\psi_\delta(w)|\,|f(x,v-w)|^2 \dd w,
\end{equation}
since $\|\nabla\psi_\delta\|_{L^1(\R^N)}=\delta^{-1}\|\nabla\psi\|_{L^1(\R^N)}$. From \eqref{boundHdeltaj0} and \eqref{noise}, we deduce that 
\begin{align}
\sum_{j\geq 0}\|D^\delta_j f\|_H^2\leq \|\nabla\psi\|_{L^1(\R^N)}^2\frac{1}{\delta^2}\|f\|_H^2,\label{boundHdeltaj1}\\
\sum_{j\geq 0}\|[D^\delta_j]^2 f\|_H\leq \|\nabla\psi\|_{L^1(\R^N)}^2\frac{1}{\delta^2}\|f\|_H.\label{boundHdeltaj2}
\end{align}
Consider now the norm $\|f\|_{E_T}=\sup_{t\in[0,T]}e^{-Mt}\left[\E\|f(t)\|_H^2\right]^{1/2}$ on $E_T$, where $M$ is suitably tuned below. Using \eqref{KtL2}, \eqref{boundHdeltaj1} and the It\={o} isometry, we have
\begin{equation}\label{FixedPointSto}
\left\|\sum_{j\geq 0}\int_0^\cdot K_{\cdot-s}^\# D^\delta_j f(s) \dd\beta_j(s)\right\|_{E_T}^2
\leq \|\nabla\psi\|_{L^1(\R^N)}^2\frac{1}{\delta^2}\frac{1}{2M+N}\|f\|_{E_T}^2
\end{equation}
provided that $M\geq N/2$. By \eqref{KtL2}, \eqref{boundHdeltaj2}, provided that $M\geq N/2$, we obtain also 
\begin{equation}\label{FixedPointDrift}
\left\|\sum_{j\geq 0}\int_0^\cdot K_{\cdot-s}^\# (\left[D^\delta_j\right]^2  f(s))\dd s\right\|_{E_T}
\leq \|\nabla\psi\|_{L^1(\R^N)}^2\frac{1}{\delta^2}\frac{2}{2M+N}\|f\|_{E_T}.
\end{equation}
Hence for $M$ large enough, the map $\mathcal{I}^\delta$ is a strict contraction on $E_T$. Therefore the result stems from the Banach fixed-point Theorem. \qed\bigskip

\begin{proposition}[Weak solutions to \eqref{FPdelta}] Let $f_\mathrm{in}\in H$. Let $f^\delta\in \mathcal{C}([0,T];L^2(\Omega;H))$ be the mild solution to \eqref{FPdelta} with initial datum $f_\mathrm{in}$. Then $f^\delta$ is a weak solution to \eqref{FPdelta} on $[0,T]$ in the sense that for all $\varphi$ in $\mathcal{C}^\infty_c(\T^N\times\R^N)$ and all $t\geq 0$,
\begin{equation}\label{SolutionFPdelta}
\begin{array}{lll}
& \displaystyle \langle f^\delta(t),\varphi\rangle  &=\ \displaystyle \langle f_\mathrm{in},\varphi\rangle  + \int_0^t \langle f^\delta(s), v\cdot\nabla_x \varphi\rangle \dd s + \lambda\sum_{j\geq 0}\int_0^t \left\langle f^\delta(s),F_j^\delta\cdot J^\delta\nabla_v \varphi\right\rangle \dd\beta_j(s) \\
&& \displaystyle+\int_0^t\langle f^\delta(s),\cQ^*(\varphi)\rangle \dd s +\frac{\lambda^2}{2}\sum\limits_{j\geq 0} \int_0^t \left\langle  f^\delta(s),\left(F_j^\delta\cdot J^\delta\nabla_v \right)^2\varphi\right\rangle \dd s, \quad \text{a.s.},
\end{array}
\end{equation}
where $\cQ^*$ is defined by \eqref{AdjointQ}. 
\label{prop:weakfdelta}\end{proposition}

\textbf{Proof of Proposition~\ref{prop:weakfdelta}.} We apply \cite[Theorem~6.5]{daprato} to obtain \eqref{SolutionFPdelta}. This means that we interpret \eqref{FPdelta} as \cite[Eq.~(6.1)]{daprato}. The correspondence in notation is as follows. In \cite[Paragraph~6.1]{daprato}, the letter $H$ denotes the space $L^2(\T^N\times\R^N)$ (consistently with our notational convention), $U=H$, $Q$ is the identity, $U_1$ is any Hilbert space containing $U=H$ with Hilbert-Schmidt embedding, $U_0=U$, $W(t)=\sum_{j\geq 0}\beta_j(t) \eps_j$, where $(\eps_j)$ is a Hilbert basis of $H$. The letter $X$ stands for $f^\delta$, the operator $A$ is the transport-Fokker-Planck operator $\L_\mathrm{FP}$ of \eqref{kineticFP}, the source term is 
$$
f(t):=\frac{\lambda^2}{2}\sum_{j\geq 0}\left[D^\delta_j\right]^2  f^\delta(t),
$$
the operator $B$ with domain $D(B)=H$ is given by 
$$
B(X)\eps_j=-\lambda D^\delta_j X,\quad X\in H, j\geq 0,
$$ 
where $D^\delta_j$ is defined in \eqref{defHdeltaj}. In particular, the condition 
$$
\E\int_0^T\|B(X(s))\|_{L^0_2}^2 \dd s<+\infty
$$
of \cite[Theorem~6.5]{daprato} is satisfied by \eqref{boundHdeltaj1} since
$$
\E\int_0^T\|B(X(s))\|_{L^0_2}^2 \dd s=\lambda^2\E\int_0^T\sum_{j\geq 0}\|D^\delta_j f^\delta(s)\|_H^2 \dd s.
$$
This yields \eqref{SolutionFPdelta}. \qed\bigskip

\begin{proposition}[Regularity of solutions to \eqref{FPdelta}] Let $f^\delta_\mathrm{in}\in H$. Let $f^\delta$ be the unique mild solution in $\mathcal{C}([0,T];L^2(\Omega;H))$ to \eqref{FPdelta} with initial datum $f^\delta_\mathrm{in}$. Then, for every $k\in\N^*$, regularity $W^{k,2}(\T^N\times\R^N)$ is propagated in the sense that
\begin{equation}\label{SobolevFP}
\sup_{t\in[0,T]}\E\|f^\delta(t)\|^2_{W^{k,2}(\T^N\times\R^N)}\leq C_k(\delta)\|f^\delta_\mathrm{in}\|^2_{W^{k,2}(\T^N\times\R^N)},
\end{equation}
where the constant $C_k(\delta)$ depends only on $\delta$, $T$, $k$, $\psi$ and $N$.
\label{prop:regfdelta}\end{proposition}

\textbf{Proof of Proposition~\ref{prop:regfdelta}.} We consider the case $k=1$ only, the proof of \eqref{SobolevFP} for higher-order regularity being completely similar. Note first that when $f\in W^{1,2}(\T^N\times\R^N)$, from \eqref{explicitKt} and \eqref{actionKt}, we derive for any $t\in[0,T]$ $\nabla_x K_t^\# f=K_t^\#\nabla_x f$ and
$$
\nabla_v K_t^\# f=e^t K_t^\#\nabla_v f+(1-e^t)K_t^\#\nabla_x f.
$$
Our starting point is the iteration scheme $f^\delta_0=f^\delta_{\mathrm{in}}$, $f^\delta_{m+1}=\mathcal{I}^\delta(f^\delta_m)$, where $\mathcal{I}^\delta$ is defined in \eqref{mildFPdeltaH}. The sequence $(f^\delta_m)$ converges to $f^\delta$ as $m\to\infty$ in the space $E_T$ used above. By differentiating the above scheme and using variants of estimates \eqref{FixedPointSto}-\eqref{FixedPointDrift} we obtain the differential inequality, $0\leq t\leq T$, $m\geq0$,
\begin{equation}\label{preGronwallREG}
\E\|f_{m+1}^\delta(t)\|_{W^{1,2}(\T^N\times\R^N)}^2\leq C(\delta)\left[\|f^\delta_\mathrm{in}\|^2_{W^{1,2}(\T^N\times\R^N)}
+\int_0^t \|f_m^\delta(s)\|_{W^{1,2}(\T^N\times\R^N)}^2\dd s\right],
\end{equation}
for a constant $C(\delta)$ depending on $\delta$, $T$, $\psi$ and $N$. This proves recursively that for any $m$ and $0\leq t\leq T$
$$
\E\|\nabla_{x,v}f_m^\delta(t)\|_{W^{1,2}(\T^N\times\R^N)}^2
\leq C(\delta)\|f^\delta_\mathrm{in}\|^2_{W^{1,2}(\T^N\times\R^N)} \sum_{p=0}^{m}\frac{(C(\delta) t)^p}{p!}\,.
$$
By lower semi-continuity of $\|\nabla_{x,v}(\cdot)\|_{E_T}$ on $E_T$, taking the limit $m\to\infty$ yields \eqref{SobolevFP} for $k=1$ with $C_1(\delta)=C(\delta)e^{TC(\delta)}$. \qed\bigskip

\subsection{Existence of weak solutions}\label{app:thmexFPexist}

Our goal is now to take the limit $\delta\to0$ and prove the existence of a solution to \eqref{FP}. Our first step provides bounds uniform with respect to $\delta$.

\begin{proposition}[Uniform bounds on solutions to \eqref{FPdelta}] Let $f^\delta_\mathrm{in}\in H$ satisfy 
\begin{equation}\label{indeltareg}
f^\delta_\mathrm{in}\in W^{k_0,2}(\T^N\times\R^N),
\end{equation} 
with a degree of regularity $k_0>2+N$. Let $f^\delta$ be the unique mild solution in $\mathcal{C}([0,T];L^2(\Omega;H))$ to \eqref{FPdelta} with initial datum $f^\delta_\mathrm{in}$. Then $f^\delta$ satisfies the following energy estimate
\begin{equation}\label{eq:Udelta}
\sup_{t\in[0,T]}\E\|f^\delta(t)\|_{L^2(\T^N\times\R^N)}^2+\E\|\nabla_v f^\delta\|_{L^2(\T^N\times[0,T]\times\R^N)}^2
\leq C\|f^\delta_\mathrm{in}\|_{L^2(\T^N\times\R^N)}^2
\end{equation}
where $C$ depends only on $T$ and $N$. Furthermore,
\begin{equation}\label{eq:Udelta2}
\E\|f^\delta\|_{C_w^\alpha([0,T];L^2(\T^N\times\R^N))}^2\leq C\|f^\delta_\mathrm{in}\|_{L^2(\T^N\times\R^N)}^2,
\end{equation}
where the constant $C$ depends only on $T$ and $N$, and where the norm $\|\cdot\|_{\mathcal{C}_w^\alpha([0,T];L^2(\T^N\times\R^N))}$ is exactly as defined in \eqref{defCwalpha}.
\label{prop:Udelta}\end{proposition}

\textbf{Proof of Proposition~\ref{prop:Udelta}.} Our choice of $k_0$ ensures a continuous embedding of the Sobolev space $W^{k_0,2}(\T^N\times\R^N)$ in $\mathcal{C}^2(\T^N\times\R^N)$. By Proposition~\ref{prop:regfdelta} and \eqref{indeltareg}, $f^\delta$ has a modification still denoted $f^\delta$ such that $f^\delta(t)$ is of class $\mathcal{C}^2(\T^N\times\R^N)$ for all $t\in[0,T]$ and such that \eqref{FPdelta} is satisfied for all $(x,v)$, almost surely. At this level of regularity we may apply It\={o}'s Formula to obtain \emph{point-wise}
\begin{align}
|f^\delta(t)|^2\quad=&\quad|f_\mathrm{in}^\delta(t)|^2\quad-
\quad2\lambda\,\sum_{j\geq 0}\int_0^t f^\delta(s) D^\delta_j f^\delta(s)\,\dd\beta_j(s)\nonumber\\
&+\quad\int_0^t \left(-\tfrac12v\cdot\nabla_x |f^\delta(s)|^2+\cQ(|f^\delta(s)|^2)-2|\nabla_v f^\delta(s)|^2+N|f^\delta(s)|^2\right)\,\dd s\label{fdeltasquare}\\
&
+\lambda^2\,\sum_{j\geq 0}\int_0^t \left(f^\delta(s) [D^\delta_j]^2 f^\delta(s)+|D^\delta_jf^\delta|^2(s)\right)\,\dd s.
\nonumber
\end{align}
Note that
$$
\iint_{\T^N\times\R^N} f(x,v) [D^\delta_j]^2 f(x,v)\dd x\dd v=-\iint_{\T^N\times\R^N} [D^\delta_j f(x,v)]^2 \dd x\dd v
$$
since $\psi$ is radially symmetric. Consequently \eqref{fdeltasquare} gives, by integration over $(\omega,x,v)$,
$$
\sup_{t\in[0,T]}e^{-Nt}\E\|f^\delta(t)\|_{L^2(\T^N\times\R^N)}^2+
2\int_0^T e^{-Nt}\E\|\nabla_v f^\delta(t)\|_{L^2(\T^N\times\R^N)}^2\dd t
\leq\|f^\delta_\mathrm{in}\|_{L^2(\T^N\times\R^N)}^2
$$
which yields \eqref{eq:Udelta}. We deduce the H\"older estimate \eqref{eq:Udelta2} from the energy estimate \eqref{eq:Udelta} and the weak formulation  \eqref{SolutionFPdelta}. As the method is the same as the one used to obtain the weak continuity in time in the proof of Theorem~\ref{thmex}, we skip details here and refer the reader to arguments leading from \eqref{supEgm} to \eqref{CwHolder}. \qed\bigskip

Then we consider a sequence of initial data $f^\delta_\mathrm{in}$, satisfying the regularity hypothesis \eqref{indeltareg}, that converges to $f_\mathrm{in}$ in $H$ as $\delta\to0$. Let $f^\delta$ be the associated mild solution to \eqref{FPdelta} (\textit{cf.} Proposition~\ref{prop:Efdelta}). The argument proving the convergence of $f^\delta$ in $L^2(\Omega;\mathcal{C}^\alpha_w([0,T];L^2(\T^N\times\R^N)))$ to a weak solution of \eqref{FP} is completely similar to the one given in the existence part of the proof of Theorem~\ref{thmex} and therefore we omit it. Let us only mention that we take $F^\delta_j=F_j*\psi_\delta$, $\psi_\delta$ being viewed as a function of $x$ here, and we use the elementary estimate
$$
\sum_{j\geq 0}\|F_j^\delta-F_j\|^2_{\infty}
\leq \sum_{j\geq 0}\omega_{L^\infty}(F_j,\delta;\T^N)^2
\leq \delta^2\sum_{j\geq 0}\|\nabla_xF_j\|^2_{\infty}
\stackrel{\delta\to0}{\rightarrow}0
$$
that follows from \eqref{noise}. To establish the latter claim we have introduced the notion of  $L^p$ modulus of continuity $\omega_{L^p}(a,\delta;K)$ for $p\in[1,+\infty]$ and any measurable set $K$, defined by
\begin{equation}\label{omegaa}
\omega_{L^p}(a,\delta;K)=\sup\left\{\|a(\cdot)-a(\cdot+Y)\|_{L^p(K)}; |Y|<\delta\right\}.
\end{equation}

\subsection{Uniqueness of weak solutions}\label{app:thmexFPuniq}

We derive uniqueness of solutions to \eqref{FP} as a consequence of \eqref{EEstimateFPsolutions}. To establish \eqref{EEstimateFPsolutions} for weak solutions, we follow a procedure similar to the one in \cite[Appendix~A.20]{villani}, using cut-off functions and convolution kernels to localize and regularize the solution and performing relevant estimates on localized regularizations then obtaining the claimed \eqref{EEstimateFPsolutions} by taking suitable limits. 

Again we use $\psi_\delta$, $\delta>0$, defined in Section~\ref{app:thmexFPapprox} as an approximation of identity for the convolution. Also we choose a cut-off function $\chi\in \mathcal{C}^\infty_c(\R^N)$ such that $0\leq\chi\leq 1$, $\chi\equiv 1$ on $B(0,1)$ and with support in the ball $B(0,2)$. Let $\delta>0$ and $\eps>0$. We denote by $\chi_\eps$ a rescaled cut-off function obtained through
$$
\chi_\eps(v)=\chi(\eps v).
$$
In what follows the generic variable in $\T^N\times\R^N$ is denoted by $X=(x,v)$ and we also abuse slightly notational conventions by sometimes seeing functions of $x$ or $v$ only as functions of $X$, for instance we use indifferently $\chi_\eps(X)$ for $\chi_\eps(v)$.
We denote by $\psi_\delta^\otimes$ the kernel on $\R^N\times\R^N$ defined by
\begin{equation}\label{L1rescale}
\psi_\delta^\otimes(X)=\frac{1}{\delta^{2N}}\psi\left(\frac{x}{\delta}\right)\psi\left(\frac{v}{\delta}\right)=\psi_{\delta}(x)\psi_{\delta}(v).
\end{equation}
In contrast with \cite[Appendix~A.20]{villani}, we use the same regularization parameter $\delta$ in space and in velocity. This slight simplification follows from a different treatment of commutators (\textit{cf.} \eqref{smallcommutator}). We denote by $J_\delta$ the operator of convolution with $\psi^\otimes_\delta$. Since $\psi^\otimes$ is symmetric, $J_\delta$ is self-adjoint on $L^2(\T^N\times\R^N)$. Moreover since $\psi^\otimes$ is smooth and compactly supported, $J_\delta$ acts continuously on $\mathcal{C}^\infty_c(\T^N\times\R^N)$ and maps $\mathcal{D}'(\T^N\times\R^N)$ on $\mathcal{C}^\infty(\T^N\times\R^N)$. In what follows we use $g_{\eps,\delta}$ as a shorthand for $\chi_\eps J_\delta(g)$ --- a localized and regularized version of $g$ --- and $g^{\eps,\delta}$ for $J_\delta(\chi_\eps g)$.

When $\varphi\in \mathcal{C}^\infty_c(\T^N\times\R^N)$, we may use  $\varphi^{\eps,\delta}$  as a test function in \eqref{SolutionFP} to obtain an equation of the form
\begin{equation}\label{preItofpesdelta}
\< f_{\eps,\delta}(t),\varphi\> = \< [f_\mathrm{in}]_{\eps,\delta},\varphi\> +\int_0^t \<g_{\eps,\delta}(s),\varphi\> \dd s
+\sum_{j\geq 0}\int_0^t \<h^j_{\eps,\delta}(s),\varphi\> \dd\beta_j(s),
\end{equation}
where, by properties of $f$, both
$$
t\mapsto \<g_{\eps,\delta}(t),\varphi\>
\qquad\textrm{ and }\qquad
t\mapsto\<h^j_{\eps,\delta}(t),\varphi\>
$$
are adapted and a.s. continuous on $[0,T]$. Indeed \eqref{preItofpesdelta} holds with $h^j_{\eps,\delta}(s)$ given by
\begin{equation}\label{hjepsdelta}
h^j_{\eps,\delta}(s)=-\lambda\,
(F_j\cdot\nabla_v f(s))_{\eps,\delta}
\end{equation}
and
\begin{equation}\label{gepsdelta}
g_{\eps,\delta}(s)=
[\div_v(vf(s))]_{\eps,\delta}-(v\cdot\nabla_xf(s))_{\eps,\delta}
+(\Delta_v f(s))_{\eps,\delta}
+\frac{\lambda^2}{2}\sum_{j\geq0}((F_j\cdot\nabla_v)^2(f(s)))_{\eps,\delta}\,.
\end{equation}
Applying It\={o}'s Formula to \eqref{preItofpesdelta} and summing over $\varphi\in\mathcal{B}$, where $\mathcal{B}$ is a Hilbert basis of $L^2(\T^N\times\R^N)$ constituted of elements of $\mathcal{C}^\infty_c(\T^N\times\R^N)$, leads \emph{via} Parseval's identity to
\begin{equation}\label{postItofpesdeltasum}
\frac12\E\| f_{\eps,\delta}(t)\|^2 = \frac12\E\|[f_\mathrm{in}]_{\eps,\delta}\|^2 +\E\int_0^t\Big[ \<g_{\eps,\delta}(s),f_{\eps,\delta}(s)\> +\frac12\sum_{j\geq 0}\|h^j_{\eps,\delta}(s)\|^2\Big] \dd s.
\end{equation}

We go on with the proof of \eqref{EEstimateFPsolutions}. To do so, we need to prove that in a certain limit $\delta\to0$, $\eps\to0$, we recover skew-symmetry of divergence-free vector fields, positivity of $-\Delta_v$, cancellation of Stratonovich terms... The most harmless corrective terms, vanishing at the limit, may be handled with the following result.

\begin{lemma} There exists $C$ such that for any $g\in L^2(\T^N\times\R^N)$ and any $\eps>0$
\begin{equation}\label{Dchieps}
\|\{\chi_\eps^2,\nabla_v\}g\|\,\leq\,C\,\eps\,\|g\|\,.
\end{equation}
There exists $C$ such that for any Lipschitz $a$ on $\T^N\times\R^N$, any $g\in L^2(\T^N\times\R^N)$ and any $\delta>0$
\begin{equation}\label{smallcommutator}
\|\{a,J_\delta\}g\|\,\leq\,C\,\delta\,\|\nabla_X a\|_{L^\infty(\T^N\times\R^N)}\,\|g\|\,.
\end{equation}
\label{lem:smallepsdelta}\end{lemma}

\textbf{Proof of Lemma~\ref{lem:smallepsdelta}.} 
Estimate~\eqref{Dchieps} is readily obtained with $C=2\|\nabla_v(\chi)\|_{L^\infty(\R^N)}$.

To obtain \eqref{smallcommutator}, we observe that
$$
\{a,J_\delta\}g(X)=\int_{\R^N\times\R^N} (a(X)-a(X-Y))\psi^\otimes_\delta(Y)  g(X-Y)\dd Y,
$$
and conclude with $C=\sqrt{2}$ by noting that $|Y|\leq\sqrt{2}\delta$ on the support of $\psi^\otimes_\delta$.\qed\bigskip

In particular, for some constant $C$, we conclude that
$$
\left|\langle(\Delta_v f(s))_{\eps,\delta},(f(s))_{\eps,\delta}\rangle+\|(\nabla_vf(s))_{\eps,\delta}\|^2\right|\,\leq\,C\,\eps\,\|f(s)\|\,\|\nabla_v f(s)\|
$$
and
$$
\begin{array}{l}
\displaystyle\Big|\frac{\lambda^2}{2}\sum_{j\geq0}\langle((F_j\cdot\nabla_v)^2(f(s)))_{\eps,\delta},(f(s))_{\eps,\delta}\rangle
+\frac12\sum_{j\geq 0}\|h^j_{\eps,\delta}(s)\|^2\Big|\\[0.5em]
\displaystyle\qquad
\leq\ 
\frac{\lambda^2}{2}\sum_{j\geq0}|\langle(J_\delta(F_j (F_j\cdot\nabla_v)(f(s))),\{\chi_\eps^2,\nabla_v\}J_\delta(f(s))\rangle|\\
\displaystyle\qquad\qquad
+\ 
\frac{\lambda^2}{2}\sum_{j\geq0}|\langle \chi_\eps\{F_j,J_\delta\}(F_j\cdot\nabla_v)(f(s)),(\nabla_vf(s))_{\eps,\delta}\rangle|\\
\displaystyle\qquad\qquad
+\ 
\frac{\lambda^2}{2}\sum_{j\geq0}|\langle((F_j\cdot\nabla_v)(f(s)))_{\delta,\eps},\chi_\eps\{F_j,J_\delta\}\nabla_vf(s)\rangle|\\[0.5em]
\displaystyle\qquad
\leq\ C\,\eps\,\|f(s)\|\,\|\nabla_v f(s)\|
\,+\,C\,\delta\,\|\nabla_v f(s)\|^2
\end{array}
$$
by using \eqref{noise}.

The remaining terms of the right-hand side of \eqref{gepsdelta} require more care, essentially because we do not control moments in velocity nor space derivatives.

To deal with the second of those terms, we propose the following variation on the proof of estimate \eqref{smallcommutator}. Observe that for any smooth $g$ 
$$
\{v\cdot\nabla_x,J_\delta\}g(X)=
\int_{\R^N\times\R^N} w\cdot\nabla_x(\psi^\otimes_\delta)(Y)  (g(X-Y)-g(X))\dd Y,
$$
with implicit notation $Y=(y,w)$, so that
$$
\|\{v\cdot\nabla_x,J_\delta\}g\|
\leq C\,\omega_{L^2}(g,\sqrt{2}\delta;\T^N\times\R^N)
$$
with $C=\|v\cdot\nabla_x\psi^\otimes\|_{L^1(\T^N\times\R^N)}$, where $\omega_{L^2}$ is as in \eqref{omegaa}. By a classical density/semi-continuity argument this extends to any $g\in L^2(\T^N\times\R^N)$. As a consequence, 
\begin{align*}
|\langle(v\cdot\nabla_xf(s))_{\eps,\delta},(f(s))_{\eps,\delta}\rangle|
=&|\langle \chi_\eps \{v\cdot\nabla_x,J_\delta\}f(s),(f(s))_{\eps,\delta}\rangle|\\
\leq & C\,\|f(s)\|\,\omega_{L^2}(f(s),\sqrt{2}\delta;\T^N\times\R^N)
\end{align*}
for some constant $C$. This is the first bound that does not provide a quantitative convergence to zero. Yet note that the estimate is uniform with respect to $\eps$.

Concerning the last term we note that
$$
\begin{array}{l}
\displaystyle\Big|\langle([\div_v(vf(s))]_{\eps,\delta},(f(s))_{\eps,\delta}\rangle
-\frac{N}{2}\|(f(s))_{\eps,\delta}(s)\|^2\Big|\\[0.5em]
\displaystyle\qquad
=\ |\langle \{\chi_\eps J_\delta,\textrm{div}(v\,\cdot)\}f(s),(f(s))_{\eps,\delta}\rangle|\\[0.5em]
\displaystyle\qquad
\leq\ |\langle\chi_\eps\{J_\delta,v\}\cdot\nabla_vf(s),(f(s))_{\eps,\delta}\rangle|
\ +\ 
|\langle v\cdot\{\chi_\eps,\nabla_v\}J_\delta(f(s)),(f(s))_{\eps,\delta}\rangle|\\[0.5em]
\displaystyle\qquad
\leq\ C\,\delta\,\|f(s)\|\,\|\nabla_v f(s)\|
\,+\,C\,\|f(s)\|\,\,\|J_\delta f(s)\|_{L^2(\{(x,v);|v|\geq \eps^{-1}\})}
\end{array}
$$
for some constant $C$ (involving $\|v\cdot\nabla\chi\|_{L^\infty(\R^N)}$). The latter estimate is far from being uniform but is sufficient jointly with the foregoing estimates to conclude by taking first $\limsup_{\eps\to0}$ then $\limsup_{\delta\to0}$ that
$$
\frac12\E\|f(t)\|^2+\E\int_0^t\|\nabla_v f(s)\|^2 \dd s
\leq \frac12\E\|f_\mathrm{in}\|^2+\frac{N}{2}\E\int_0^t \|f(s)\|^2\dd s.
$$
thus to obtain \eqref{EEstimateFPsolutions} by an application of the Gronwall Lemma. \qed

\bibliographystyle{abbrv}
\bibliography{biblio}

\begin{thebibliography}{10}

\bibitem{CardaliaguetDelarueLasryLions2015}
P.~Cardaliaguet, F.~Delarue, J.-M. Lasry, and P.-L. Lions.
\newblock The master equation and the convergence problem in mean field games.
\newblock {\em arXiv:1509.02505 [math]}, Sept. 2015.

\bibitem{ChungWilliams90}
K.~L. Chung and R.~J. Williams.
\newblock {\em Introduction to stochastic integration}.
\newblock Probability and its Applications. Birkh\"auser Boston, Inc., Boston,
  MA, second edition, 1990.

\bibitem{daprato}
G.~Da~Prato and J.~Zabczyk.
\newblock {\em Stochastic equations in infinite dimensions}, volume~44 of {\em
  Encyclopedia of Mathematics and its Applications}.
\newblock Cambridge University Press, Cambridge, 1992.

\bibitem{FrizHairer2014}
P.~K. Friz and M.~Hairer.
\newblock {\em A Course on Rough Paths}.
\newblock Universitext. {Springer, Cham}, 2014.
\newblock With an introduction to regularity structures.

\bibitem{Gallay-Wayne-invariant_manifold}
T.~Gallay and C.~Wayne.
\newblock Invariant manifolds and the long-time asymptotics of the
  {N}avier-{S}tokes and vorticity equations on $\bf{R^2}$.
\newblock {\em Arch. Ration. Mech. Anal.}, 163(3):209--258, 2002.

\bibitem{Helffer}
B.~Helffer.
\newblock {\em Spectral theory and its applications}, volume 139 of {\em
  Cambridge Studies in Advanced Mathematics}.
\newblock Cambridge University Press, Cambridge, 2013.

\bibitem{HerauThomann16}
F.~H{\'e}rau and L.~Thomann.
\newblock On global existence and trend to the equilibrium for the
  {V}lasov-{P}oisson-{F}okker-{P}lanck system with exterior confining
  potential.
\newblock {\em J. Funct. Anal.}, 271(5):1301--1340, 2016.

\bibitem{HwangJang13}
H.~J. Hwang and J.~Jang.
\newblock On the {V}lasov-{P}oisson-{F}okker-{P}lanck equation near
  {M}axwellian.
\newblock {\em Discrete Contin. Dyn. Syst. Ser. B}, 18(3):681--691, 2013.

\bibitem{mouhotneumann}
C.~Mouhot and L.~Neumann.
\newblock Quantitative perturbative study of convergence to equilibrium for
  collisional kinetic models in the torus.
\newblock {\em Nonlinearity}, 19(4):969--998, 2006.

\bibitem{Oksendal}
B.~{\O}ksendal.
\newblock {\em Stochastic differential equations}.
\newblock Universitext. Springer-Verlag, Berlin, fifth edition, 1998.
\newblock An introduction with applications.

\bibitem{ReedSimonI}
M.~Reed and B.~Simon.
\newblock {\em Methods of modern mathematical physics. {I}}.
\newblock Academic Press, Inc. [Harcourt Brace Jovanovich, Publishers], New
  York, second edition, 1980.
\newblock Functional analysis.

\bibitem{RevuzYor99}
D.~Revuz and M.~Yor.
\newblock {\em Continuous martingales and {B}rownian motion}, volume 293 of
  {\em Grundlehren der Mathematischen Wissenschaften [Fundamental Principles of
  Mathematical Sciences]}.
\newblock Springer-Verlag, Berlin, third edition, 1999.

\bibitem{SchillingPartzsch2014}
R.~L. Schilling and L.~Partzsch.
\newblock {\em Brownian Motion}.
\newblock De Gruyter Graduate. {De Gruyter, Berlin}, second edition, 2014.
\newblock An introduction to stochastic processes, With a chapter on simulation
  by Bj{\"o}rn B{\"o}ttcher.

\bibitem{villani}
C.~Villani.
\newblock Hypocoercivity.
\newblock {\em Mem. Amer. Math. Soc.}, 202(950):iv+141, 2009.

\bibitem{VillaniOldNew}
C.~Villani.
\newblock {\em Optimal transport}, volume 338 of {\em Grundlehren der
  Mathematischen Wissenschaften [Fundamental Principles of Mathematical
  Sciences]}.
\newblock Springer-Verlag, Berlin, 2009.
\newblock Old and new.

\bibitem{WongZakai1965}
E.~Wong and M.~Zakai.
\newblock On the {{Convergence}} of {{Ordinary Integrals}} to {{Stochastic
  Integrals}}.
\newblock {\em The Annals of Mathematical Statistics}, 36(5):1560--1564, Oct.
  1965.

\end{thebibliography}

\end{document}